\documentclass[]{article}

\usepackage{amsmath,amsthm,amssymb,mathrsfs,graphicx,subfigure,extarrows,float}
\usepackage[numbers]{natbib}
\usepackage{titlesec}
\usepackage[titletoc,toc,title]{appendix}
\usepackage{graphicx}
\usepackage{faktor}
\usepackage[a4paper, textheight=9.6in]{geometry}

\title{Diagrammatic morphisms between indecomposable modules of $\bar{U}_{q}(\mathfrak{sl}_{2})$}
\author{Stephen T. Moore \\
	\multicolumn{1}{p{.8\textwidth}}{\centering\emph{Dept. of Mathematics, Ben Gurion University, Beer-Sheva, Israel,\\ stm862@gmail.com}}}

\newtheorem{thm}{Theorem}[section]
\newtheorem{prop}{Proposition}[section]

\newtheorem{corr}{Corollary}[section]
\newtheorem{lemma}{Lemma}[section]
\newtheorem{remark}{Remark}[section]

\begin{document}

\maketitle

\begin{abstract}
We give diagrammatic formulae for morphisms between indecomposable representations of $\bar{U}_{q}(\mathfrak{sl}_{2})$ appearing in the decomposition of $\mathbb{C}^{\otimes 2n}$, including projections and second endomorphisms on projective indecomposable representations.
\end{abstract}

\section{Introduction}

The Temperley-Lieb algebras, $TL_{n}(\delta)$, are an important family of algebras which link a number of areas of mathematics, including statistical mechanics, subfactors, quantum groups, and knot theory \cite{TL,Jones1,Martin,Jones2}. One way of realizing them is as the centralizer algebra of $U_{q}(\mathfrak{sl}_{2})$ on $\mathbb{C}^{\otimes 2n}$. Fundamental in this construction and in many uses of the Temperley-Lieb algebras are the Jones-Wenzl projections, $f_{n}$. In the centralizer construction, these are the projections onto the unique irreducible representation $\mathcal{X}_{n+1}$ appearing in the decomposition of $\mathbb{C}^{\otimes 2n}$. 

Recently, there has been interest in a finite dimensional quotient of $U_{q}(\mathfrak{sl}_{2})$ at even roots of unity, known as the \textit{restricted quantum group}, $\bar{U}_{q}(\mathfrak{sl}_{2})$. This algebra was conjectured in \cite{FGST3} to have representation theory equivalent to a logarithmic conformal field theory. A description of its centralizer algebra on $\mathbb{C}^{\otimes 2n}$ was given in \cite{GST,methesis,mepaper}, and its representation theory has shown potential for topological invariants based on modified traces \cite{BBG}. 

$\bar{U}_{q}(\mathfrak{sl}_{2})$ is a non-semisimple Hopf algebra, and both irreducible and projective indecomposable representations appear in the decomposition of $\mathbb{C}^{\otimes 2n}$. In Section \ref{background}, we review the representation theory of $\bar{U}_{q}(\mathfrak{sl}_{2})$, and describe its centralizer algebra, given here in the form of a planar algebra. Section \ref{projections} contains the main result of this paper, which is to give formulae for the projections onto the unique indecomposable representations appearing in the $\mathbb{C}^{\otimes 2n}$ decomposition, similar to the Jones-Wenzl projections. In Section \ref{morphisms}, we give diagrammatic formulae for other non-trivial morphisms between indecomposable representations, including their nilpotent second endomorphisms. Combinatorial relations for $\bar{U}_{q}(\mathfrak{sl}_{2})$ and its action on $\mathbb{C}^{\otimes 2n}$ are given in the Appendix. Inductive formulae for the indecomposable projections and second endomorphisms were given independently in \cite{Ibanez}.

\section{Background}\label{background}
For $q=e^{i\pi/p}$, $p\geq 2$, and $p\in \mathbb{N}$, the restricted quantum group $\bar{U}_{q}(\mathfrak{sl}_{2})$ over a field $\mathbb{K}$ is the Hopf algebra generated by $E,F,K$ subject to the relations:
\begin{align*}
KEK^{-1}&=q^{2}E& KFK^{-1}&=q^{-2}F& EF-FE&=\frac{K-K^{-1}}{q-q^{-1}}\\
E^{p}&=0& F^{p}&=0& K^{2p}&=1
\end{align*}
and coproduct $\Delta$, counit $\epsilon$, and antipode $S$:
\begin{align*}
\Delta:E&\mapsto E\otimes K+1\otimes E& F&\mapsto F\otimes 1+K^{-1}\otimes F& K&\mapsto K\otimes K\\
\epsilon:E&\mapsto 0& F&\mapsto 0& K&\mapsto 1\\
S:E&\mapsto-EK^{-1}& F&\mapsto-KF& K&\mapsto K^{-1}
\end{align*}
The modules for $\bar{U}_{q}(\mathfrak{sl}_{2})$ were given in \cite{Arike,GST,KoSa,Suter,Xiao}. The modules relevant here are the simple and projective modules, and consist of the following:
$2p-2$ simple modules, $\mathcal{X}^{\pm}_{s}$, $1\leq s<p$, two simple projective modules $\mathcal{X}^{\pm}_{p}$ and $2p-2$ non-simple indecomposable projective modules $\mathcal{P}^{\pm}_{s}$. The simple modules $\mathcal{X}^{\pm}_{s}$, $1\leq s\leq p$, have basis $\{\nu_{n}^{s}\}_{n=0,...,s-1}$ and $\bar{U}_{q}(\mathfrak{sl}_{2})$ action given by:
\begin{align*}
K\nu_{n}&=\pm q^{s-1-2n}\nu_{n}\\
E\nu_{n}&=\pm[n][s-n]\nu_{n-1}\\
F\nu_{n}&=\nu_{n+1}
\end{align*}
where $\nu_{-1}=\nu_{s}=0$ and $[n]=\frac{q^{n}-q^{-n}}{q-q^{-1}}=q^{n-1}+q^{n-3}+...+q^{3-n}+q^{1-n}$. The projective modules $\mathcal{P}^{\pm}_{s}$, $1\leq s<p$, for a given choice of $p$, can be given in terms of the basis $\{a_{i}^{s,p},b_{i}^{s,p}\}_{0\leq i \leq s-1}\cup\{x_{j}^{s,p},y_{j}^{s,p}\}_{0\leq j\leq p-s-1}$. The action of $\bar{U}_{q}(\mathfrak{sl}_{2})$ is given by:
\begin{align*}
Ka_{i}&=\pm q^{s-1-2i}a_{i}& Kb_{i}&=\pm q^{s-1-2i}b_{i}\\
Kx_{j}&=\mp q^{p-s-1-2j}x_{j}& Ky_{j}&=\mp q^{p-s-1-2j}y_{j}\\
Ea_{i}&=\pm[i][s-i]a_{i-1}& Eb_{i}&=\pm[i][s-i]b_{i-1}+a_{i-1}& Eb_{0}&=x_{p-s-1}\\
Ex_{j}&=\mp[j][p-s-j]x_{j-1}& Ey_{j}&=\mp[j][p-s-j]y_{j-1}& Ey_{0}&=a_{s-1}\\
Fa_{i}&=a_{i+1}& Fb_{i}&=b_{i+1}& Fb_{s-1}&=y_{0}\\
Fx_{j}&=x_{j+1}& Fx_{p-s-1}&=a_{0}& Fy_{j}&=y_{j+1}
\end{align*}
where $x_{-1}=a_{-1}=a_{s}=y_{p-s}=0$, \cite{FGST2}.\\
\\
The maps between indecomposable modules can be summarized as follows:
\begin{itemize}
	\item $\dim\left( Hom(\mathcal{X}^{\pm}_{s},\mathcal{X}^{\pm}_{t})\right) =0$ for $s\neq t$ or $1$ for $s=t$, for $1\leq s,t\leq p$.
	\item $ Hom(\mathcal{X}^{\pm}_{s},\mathcal{X}^{\mp}_{t}) =0$ for $1\leq s,t\leq p$.
	\item $ \dim\left( Hom(\mathcal{P}^{\pm}_{s},\mathcal{X}^{\pm}_{t})\right) =0$ for $s\neq t$ or $1$ for $s=t$, for $1\leq s,t \leq p-1$.
	\item $ Hom(\mathcal{P}^{\pm}_{s},\mathcal{X}^{\mp}_{t}) =0$ for $1\leq s,t\leq p-1$.
	\item $\dim\left( Hom(\mathcal{P}^{\pm}_{s},\mathcal{P}^{\pm}_{t})\right) =0$ for $s\neq t$ or $2$ for $s=t$, $1\leq s,t\leq p-1$.
	\item $\dim\left( Hom(\mathcal{P}^{\pm}_{s},\mathcal{P}^{\mp}_{t})\right) =0$ for $s\neq p-t$ or $2$ for $s=p-t$, for $1\leq s,t\leq p-1$.
\end{itemize}
The non-zero, and non-identity parts of these maps are given in terms of bases by:
\begin{align*}
\mathcal{P}^{\pm}_{s}\rightarrow\mathcal{X}^{\pm}_{s}:& b_{i}\mapsto \nu_{i}\\
\mathcal{X}^{\pm}_{s}\rightarrow\mathcal{P}^{\pm}_{s}:& \nu_{i}\mapsto a_{i}\\
\mathcal{P}^{\pm}_{s}\rightarrow\mathcal{P}^{\pm}_{s}:& b_{i}\mapsto a_{i}\\
\mathcal{P}^{\pm}_{s}\rightarrow\mathcal{P}^{\mp}_{p-s}:& b_{i}\mapsto g_{1}\tilde{x}_{i}+g_{2}\tilde{y}_{i}\\
& x_{j}\mapsto g_{2}\tilde{a}_{j}\\
& y_{j}\mapsto g_{1}\tilde{a}_{j}
\end{align*}
where $g_{1},g_{2}\in\mathbb{K}$, and we denote the elements of $\mathcal{P}^{\mp}_{s}$ with $\sim$.\\
\\
From now on, we denote the module $\mathcal{X}^{+}_{2}$ by $X$. The fusion rules for $\bar{U}_{q}(\mathfrak{sl}_{2})$ modules were given in \cite{KoSa,TsuWo}. The fusion rules for the simple and projective modules are:
\begin{align*}
&\mathcal{X}^{i}_{1}\otimes\mathcal{X}^{j}_{s}\simeq\mathcal{X}^{j}_{s}\otimes\mathcal{X}^{i}_{1}\simeq\mathcal{X}^{ij}_{s}, \:\:\: i,j\in\{+,-\}\\
&\mathcal{X}^{i}_{1}\otimes\mathcal{P}^{j}_{s}\simeq\mathcal{P}^{j}_{s}\otimes\mathcal{X}^{i}_{1}\simeq\mathcal{P}^{ij}_{s}\\
& X\otimes\mathcal{X}^{\pm}_{t}\simeq\mathcal{X}^{\pm}_{t}\otimes X\simeq\mathcal{X}^{\pm}_{t-1}\oplus\mathcal{X}^{\pm}_{t+1}, \:\:\: 2\leq t\leq p-1\\
& X\otimes\mathcal{X}^{\pm}_{p}\simeq \mathcal{X}^{\pm}_{p}\otimes X\simeq \mathcal{P}^{\pm}_{p-1}\\
& X\otimes\mathcal{P}^{\pm}_{p-1}\simeq \mathcal{P}^{\pm}_{p-1}\otimes X\simeq \mathcal{P}^{\pm}_{p-2}\oplus 2\mathcal{X}^{\pm}_{p}\\
& X\otimes\mathcal{P}^{\pm}_{u}\simeq\mathcal{P}^{\pm}_{u}\otimes X\simeq\mathcal{P}^{\pm}_{u-1}\oplus\mathcal{P}^{\pm}_{u+1}, \:\:\: 2\leq u\leq p-2\\
& X\otimes\mathcal{P}^{\pm}_{1}\simeq\mathcal{P}^{\pm}_{1}\otimes X\simeq\mathcal{P}^{\pm}_{2}\oplus 2\mathcal{X}^{\mp}_{p}
\end{align*}
From this, it follows that $\mathcal{X}^{+}_{s}$ first appears in the decomposition of $X^{\otimes (s-1)}$, $\mathcal{P}^{+}_{t}$ first appears in the decomposition of $X^{\otimes(2p-t-1)}$, $\mathcal{X}^{-}_{p}$ in $X^{\otimes(2p-1)}$, and $\mathcal{P}^{-}_{u}$ in $X^{\otimes(3p-u-1)}$. $\mathcal{X}^{-}_{v}$, $1\leq v\leq p-1$, does not appear at all in the decomposition of $X^{\otimes n}$.

\subsection{The $\bar{U}_{q}(\mathfrak{sl}_{2})$ Planar Algebra}
For detailed introductions to planar algebras, see \cite{EvansPugh2, JonesP1, Kuper, MPS}. The $\bar{U}_{q}(\mathfrak{sl}_{2})$ Planar Algebra is a diagrammatic description of $End_{\bar{U}_{q}(\mathfrak{sl}_{2})}(X^{\otimes n})$, or more generally $Hom(X^{\otimes n},X^{\otimes m})$. For a more detailed introduction to this planar algebra, see \cite{methesis,mepaper}.
The module $X:=\mathcal{X}^{+}_{2}$ has basis $\{\nu_{0},\nu_{1}\}$, with $\bar{U}_{q}(\mathfrak{sl}_{2})$ action:
\begin{align*}
K(\nu_{0})&=q\nu_{0} & E(\nu_{0})&=0 & F(\nu_{0})&=\nu_{1}\\
K(\nu_{1})&=q^{-1}\nu_{1} & E(\nu_{1})&=\nu_{0} & F(\nu_{1})&=0
\end{align*}
The action of $\bar{U}_{q}(\mathfrak{sl}_{2})$ on $X^{\otimes n}$ is then given by use of the coproduct.\\
\\
We denote by $\rho_{i_{1},...,i_{n},z}$ the element of $X^{\otimes z}$ with $\nu_{1}$ at positions $i_{1},...,i_{n}$, and $\nu_{0}$ elsewhere. We also occasionally omit the $\otimes$ sign, and combine indices. For example, $\rho_{1,3,5}=\nu_{1}\otimes\nu_{0}\otimes\nu_{1}\otimes\nu_{0}\otimes\nu_{0}=\nu_{10100}$. The elements of $X^{\otimes z}$ can be described in terms of the $K$-action on them. For $x\in X^{\otimes z}$, with $K(x)=\lambda x$, $\lambda\in\mathbb{K}$, we call $\lambda$ the weight of $x$. Alternatively for basis elements we can write this as $K(\rho_{i_{1},...,i_{n},z})=q^{z-2n}x$, and refer to $n$ also as the weight. $X^{\otimes z}$ will then have the set of weights $\{q^{z},q^{z-2},...,q^{2-z},q^{-z}\}$. Denoting the set of elements of $X^{\otimes z}$ with weight $q^{z-2n}$ by $X_{n,z}$, we have $X^{\otimes z}=\bigcup\limits_{i=0}^{z}X_{i,z}$. The weight spaces $X_{0,z}$, $X_{z,z}$ both have a single element, which we denote by $x_{0,z}:=(\nu_{0})^{\otimes z}$, $x_{z,z}:=(\nu_{1})^{\otimes z}$ respectively, and occasionally drop the second index if the context is clear. We have $\rho_{i_{1},...,i_{n},z}\in X_{n,z}$. We record a number of combinatorial relations involving $\bar{U}_{q}(\mathfrak{sl}_{2})$ and its action on $X^{\otimes z}$ in the Appendix.\\

It was shown in \cite{GST} that for $n<2p-1$, $End(X^{\otimes n})\simeq TL_{n}(q+q^{-1})$. The action of the $TL_{n}$ generators is as the composition of two maps, $X\otimes X\rightarrow \mathcal{X}^{+}_{1}$, and $\mathcal{X}^{+}_{1}\rightarrow X\otimes X$. In terms of $\nu_{ij}$, these are:
\begin{align*}
\cup(\nu_{00})=&\cup(\nu_{11})=0,& \cup(\nu_{01})=&-q\nu,& \cup(\nu_{10})=& \nu,& \cap(\nu)=& q^{-1}\nu_{10}-\nu_{01}
\end{align*}
These generalize to give maps $\cup_{i}$, $\cap_{i}$ acting on the $i$th and $(i+1)$th positions of $X^{\otimes n}$, which then give $\mathbf{e}_{i}=\cap_{i}\cup_{i}$.\\
The \textit{Jones-Wenzl projections} \cite{Wenzl}, are the unique projections $X^{\otimes n}\rightarrow\mathcal{X}^{+}_{n+1}\rightarrow X^{\otimes n}$. They are defined inductively by $f_{0}:=\Box $, $f_{1}:=\lvert$, $f_{n}:=f_{n-1}\otimes\lvert -\frac{[n-1]}{[n]}f_{n-1}\mathbf{e}_{n}f_{n-1}$. Explicitly, this projection is given by $\rho_{i_{1},...,i_{k},n}\mapsto q^{\big(kn-\frac{1}{2}(k^{2}-k)-(\sum\limits_{j=1}^{k}i_{j})\big)}\frac{([n-k]!)}{([n]!)}F^{k}x_{0,n}$.\\

For $n\geq 2p-1$, The planar algebra has extra generators $\alpha_{i},\beta_{i}$, $1\leq i\leq n-2p+2$, which act on $2p-1$ copies of $X$. Explicitly, these generators are given by:
\begin{align*}
\alpha(\rho_{i_{1},...,i_{k},2p-1})&:=q^{\big(k(2p-1)-\frac{1}{2}(k^{2}-k)-(\sum\limits_{j=1}^{k}i_{j})\big)}([k]!)E^{p-k-1}x_{2p-1}\\
\beta(\rho_{i_{1},...,i_{k},2p-1})&:=q^{\big(k(2p-1)-\frac{1}{2}(k^{2}-k)-(\sum\limits_{j=1}^{k}i_{j})\big)}([2p-1-k]!)F^{k-p}x_{0}
\end{align*}
where $E^{-1}=F^{-1}=0$. In terms of weight spaces, the generators act as:
\begin{align*}
\alpha:& X_{k,2p-1}\rightarrow X_{k+p,2p-1},& \beta:& X_{k,2p-1}\rightarrow X_{k-p,2p-1} 
\end{align*}
The relations satisfied by these generators were proven in \cite{GST, mepaper}, and are given as follows:
\begin{thm}\label{thm}
	The generators, $\alpha_{i}$ and $\beta_{i}$, satisfy the following properties:
	\begin{align}
	\alpha^{2}&=\beta^{2}=0 \label{eq:1}\\
	\alpha\beta\alpha&=\gamma\alpha \label{eq:2}\\
	\beta\alpha\beta&=\gamma\beta \label{eq:3}\\
	\gamma&=(-1)^{p-1}([p-1]!)^{2} \nonumber \\
	\alpha_{i}\alpha_{j}=\alpha_{j}\alpha_{i}&=\beta_{i}\beta_{j}=\beta_{j}\beta_{i}=0, \:\: \lvert i-j\rvert< p \label{eq:4}\\
	\alpha_{i}\alpha_{i+p}&=\alpha_{i+p}\alpha_{i} \label{eq:5}\\
	\beta_{i}\beta_{i+p}&=\beta_{i+p}\beta_{i} \label{eq:6}\\
	\alpha\beta+\beta\alpha&=\gamma f_{2p-1} \label{eq:7}
	\end{align}
	Denote by $R_{n}$ the (clockwise) (n,n)-point annular rotation tangle. We then have:
	\begin{align}
	\alpha\cap_{i}=\cup_{i}\alpha&=\beta\cap_{i}=\cup_{i}\beta=0, \:\: 1\leq i\leq 2p-2 \label{eq:8}\\
	\alpha_{i+1}\cap_{i}&=\alpha_{i}\cap_{i+2p-2} \label{eq:9}\\
	\beta_{i+1}\cap_{i}&=\beta_{i}\cap_{i+2p-2} \label{eq:10}\\
	\cup_{i}\alpha_{i+1}&=\cup_{i+2p-2}\alpha_{i} \label{eq:11}\\
	\cup_{i}\beta_{i+1}&=\cup_{i+2p-2}\beta_{i} \label{eq:12}\\
	R_{4p-2}(\alpha)&=\alpha \label{eq:13}\\
	R_{4p-2}(\beta)&=\beta \label{eq:14}\\
	\sum\limits_{i=0}^{4p-1} k_{i}R^{i}_{4p}(\alpha\otimes 1)&=0 \label{eq:15}\\
	\sum\limits_{i=0}^{4p-1} k_{i}R^{i}_{4p}(\beta\otimes 1)&=0 \label{eq:16}
	\end{align}
	where $k_{i}=(-1)^{i}[i-2]k_{1}+(-1)^{i}[i-1]k_{2}$, for arbitrary $k_{1},k_{2}\in\mathbb{K}$.
\end{thm}
It was conjectured in \cite{mepaper} that these generators and relations fully described the $\bar{U}_{q}(\mathfrak{sl}_{2})$ planar algebra. \cite{mepaper} also gives diagrammatic versions of the above relations.

\section{The Projections onto Indecomposable Representations}\label{projections}
We denote the projections onto $\mathcal{P}^{+}_{i}$ and $\mathcal{P}^{-}_{i}\oplus\mathcal{P}^{-}_{i}$ by $P^{+}_{i}$ and $P^{-}_{i}\oplus P^{-}_{i}$ respectively. For all diagrams in the following we use thick lines to denote multiple strings. The aim of this section is to prove the following:
\begin{thm}\label{ind proj}
	The projections onto $\mathcal{P}^{+}_{i}$ and $\mathcal{P}^{-}_{i}\oplus\mathcal{P}^{-}_{i}$ for $1\leq i\leq p-1$ can be given diagrammatically by the following:
	\begin{figure}[H]
		\centering
		\includegraphics[width=0.5\linewidth]{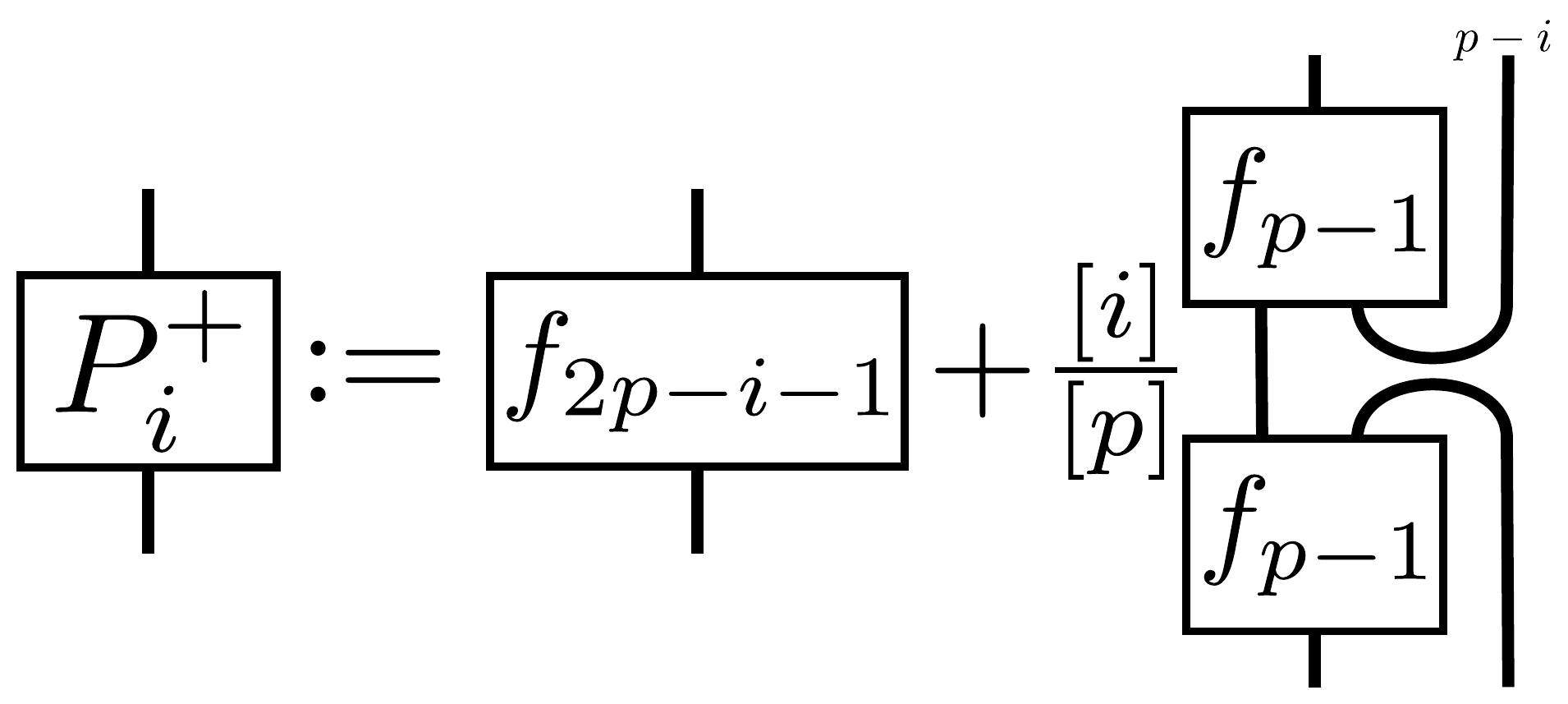}
	\end{figure}
	\begin{figure}[H]
		\centering
		\includegraphics[width=0.6\linewidth]{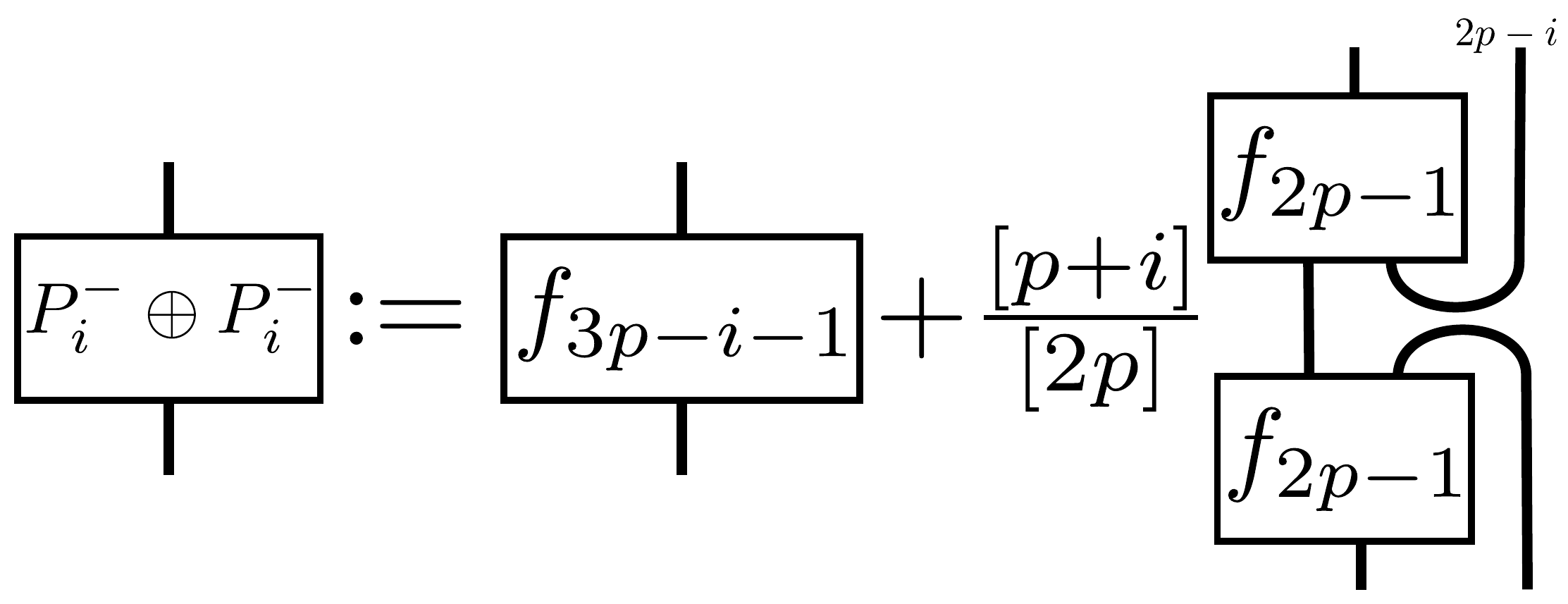}
	\end{figure}
	Relations \ref{eq:4} and \ref{eq:7} allow the projection onto $\mathcal{P}^{-}_{i}\oplus\mathcal{P}^{-}_{i}$ to be split into two orthogonal projections on $\mathcal{P}^{-}_{i}$.
\end{thm}

We give a proof for the positive case, with the negative case following similarly. To prove this, we first need  the following lemma:
\begin{lemma}
	For $1\leq i\leq z$, and $\delta$ arbitrary, the following diagram does not have $[z+1]$ appearing in the denominator of the coefficient of any element:
	\begin{figure}[H]
		\centering
		\includegraphics[width=0.5\linewidth]{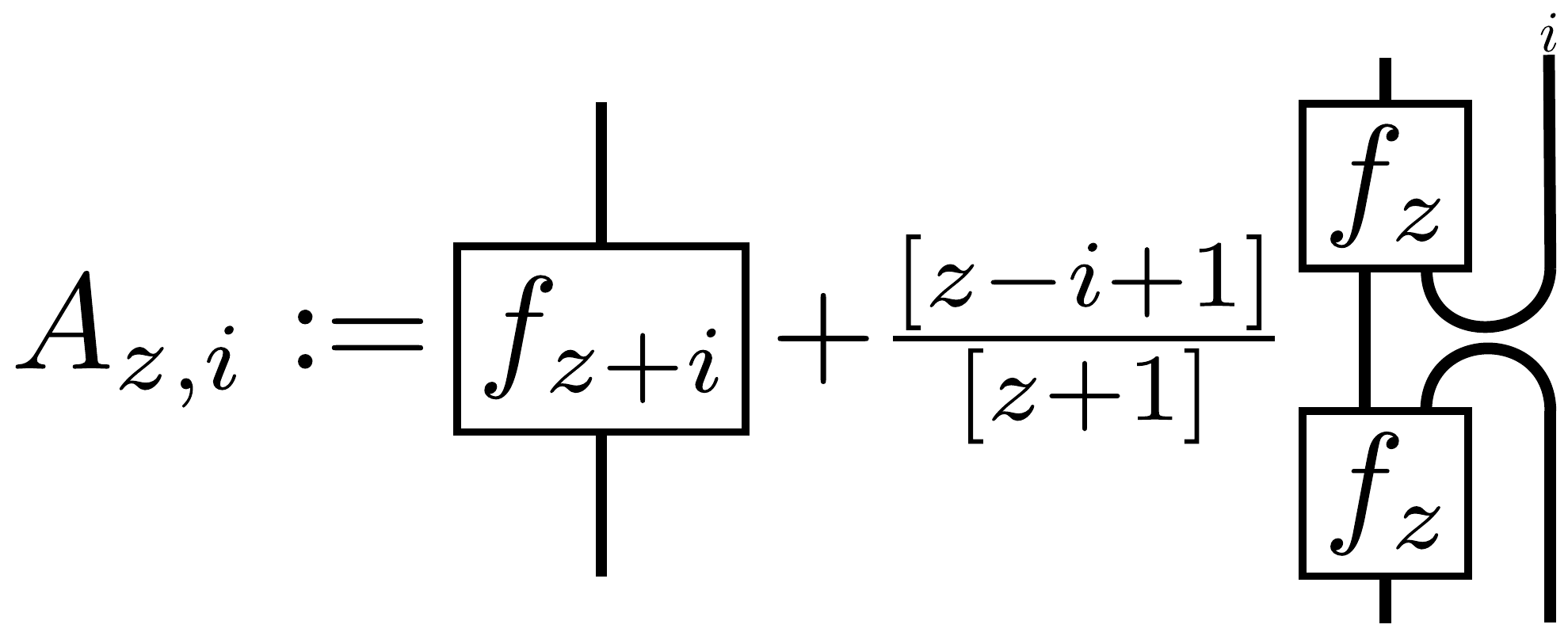}
	\end{figure}
\end{lemma}	
\begin{proof}	
	We note first that applying cups or caps to $A_{z,i}$ doesn't necessarily give zero, instead we have:
	\begin{figure}[H]
		\centering
		\includegraphics[width=0.7\linewidth]{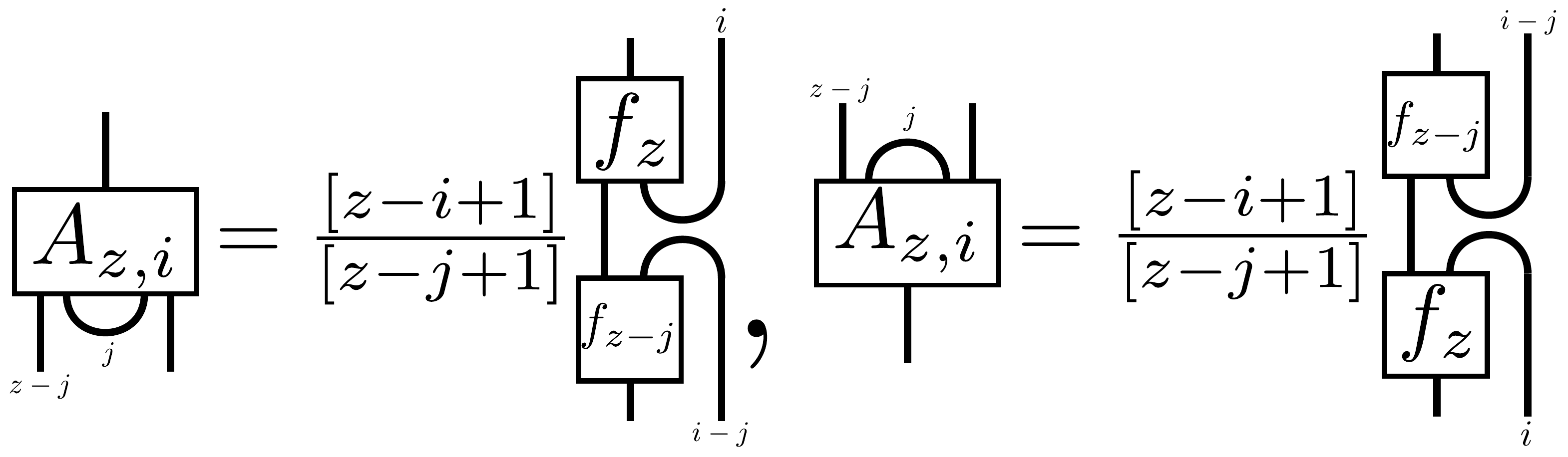}
	\end{figure}
	We proceed by induction. The cases $i=1,2$, can be shown directly by expanding $f_{z+i}$ in terms of $f_{z}$ and simplifying, as shown in Section $6$ of \cite{methesis}. Assume now that it is true for $i$, then rewriting in terms of $f_{z+i}$ and using the inductive formula for Jones-Wenzl projections, we get:
	\begin{figure}[H]
		\centering
		\includegraphics[width=1\linewidth]{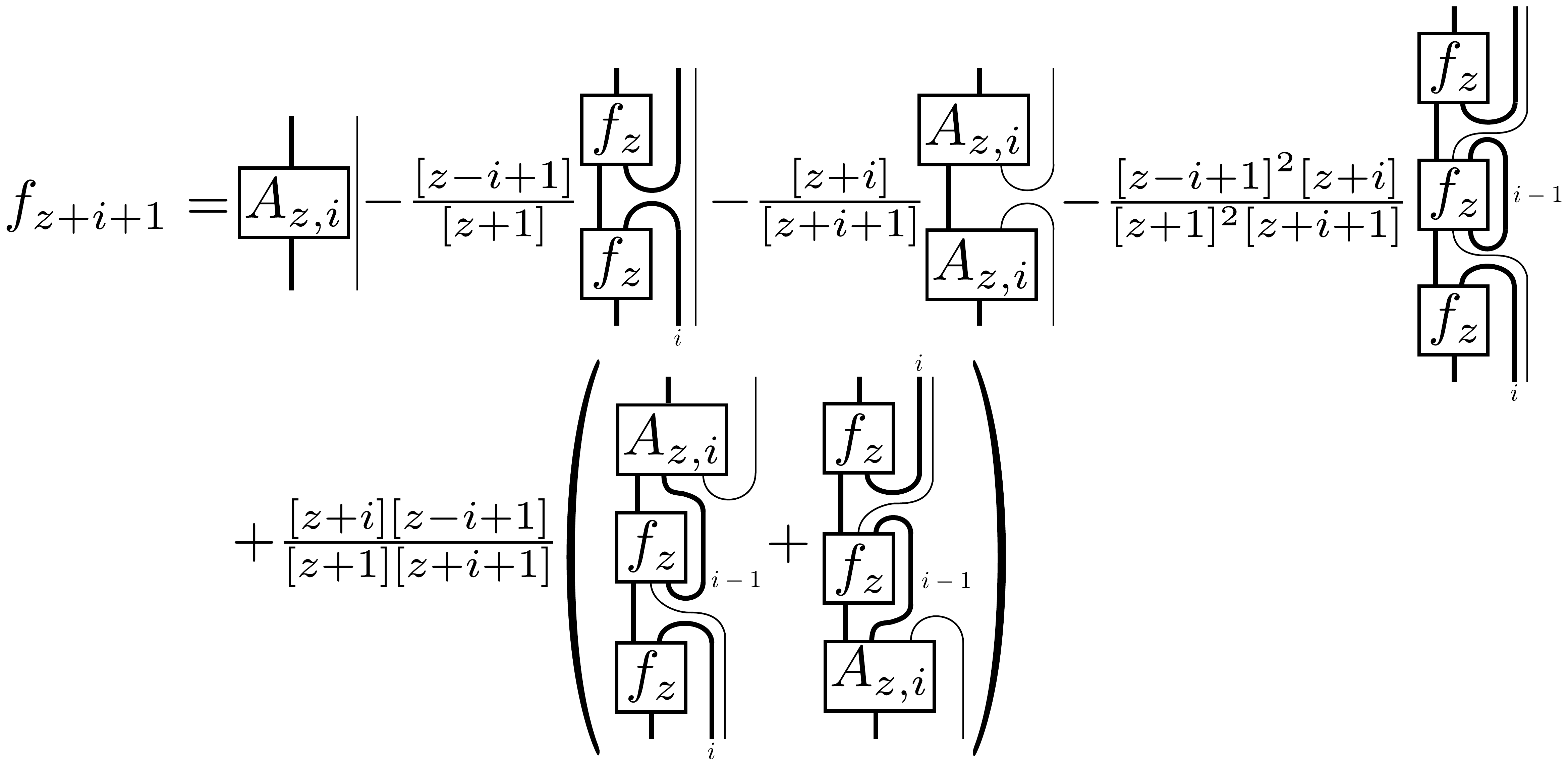}
	\end{figure}
	Simplifying this, we get:
	\begin{figure}[H]
		\centering
		\includegraphics[width=1\linewidth]{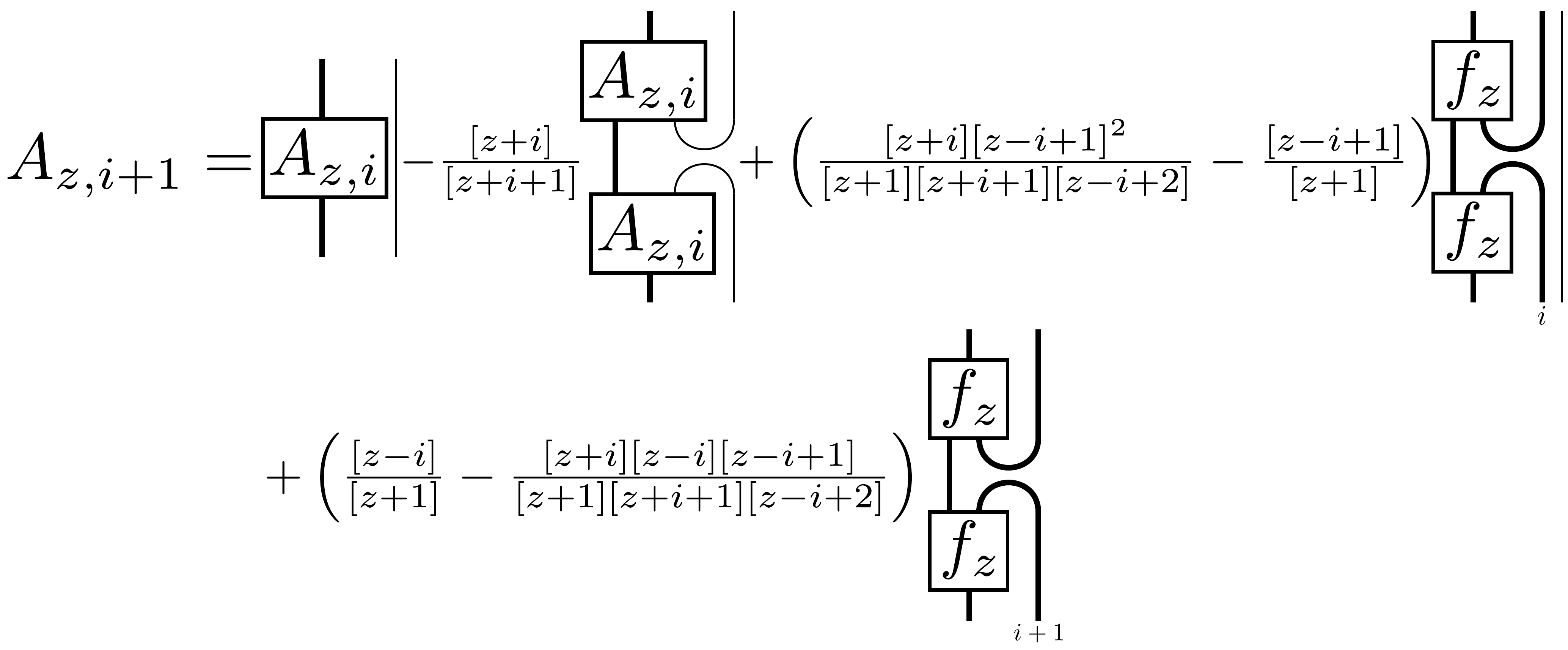}
	\end{figure}
	We now want to simplify the coefficients of the above diagram. We have:
	\begin{align*}
	\frac{([z+i][z-i+1]-[z-i+2][z+i+1])}{[z+1]}&=\frac{-(q^{2z+2}-q^{-2z-2})}{q^{z+1}-q^{-z-1}}=-q^{z+1}-q^{-z-1}
	\end{align*}
	Hence we have:
	\begin{figure}[H]
		\centering
		\includegraphics[width=1\linewidth]{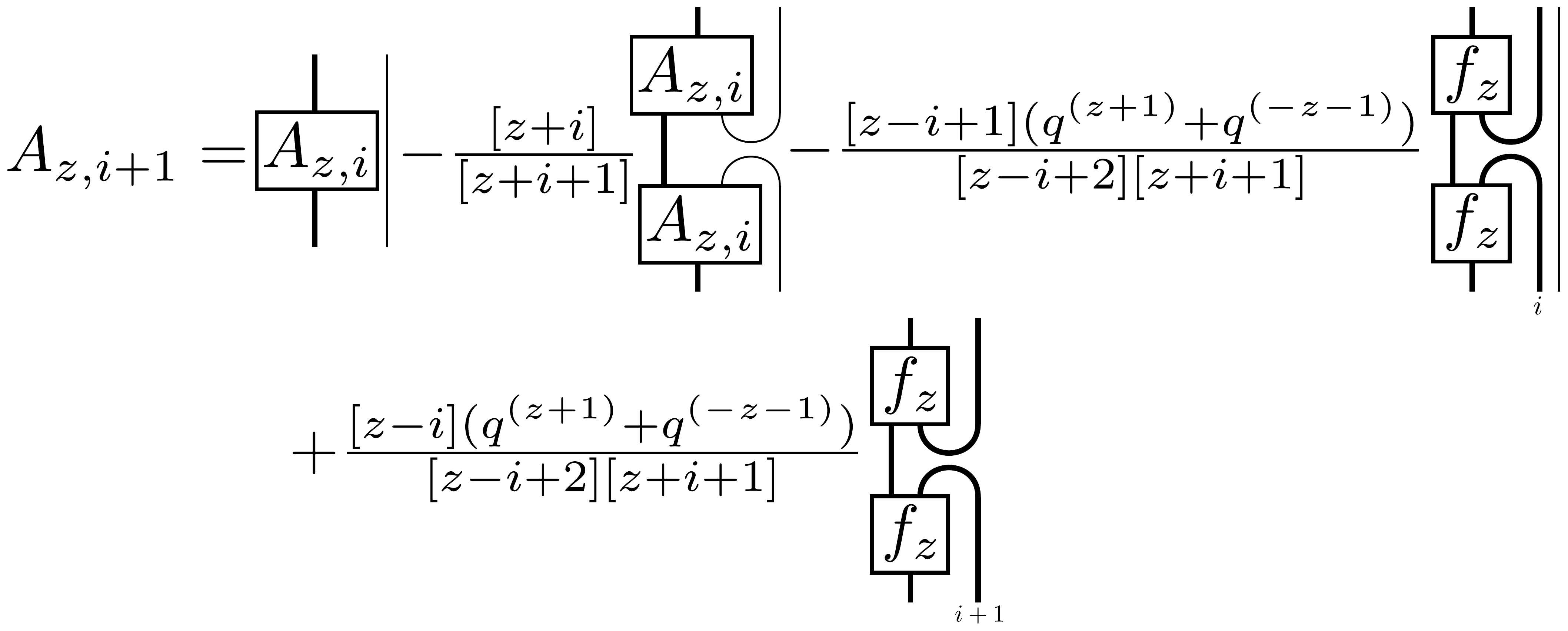}
	\end{figure}
	which doesn't contain $[z+1]$ in the denominator.	
\end{proof}
Note that this inductive formula for $A_{z,i+1}$ only holds for $i\geq 2$. Using the above result, we get the following inductive formula for the projections for $i\leq p-2$:
\begin{figure}[H]
	\centering
	\includegraphics[width=0.9\linewidth]{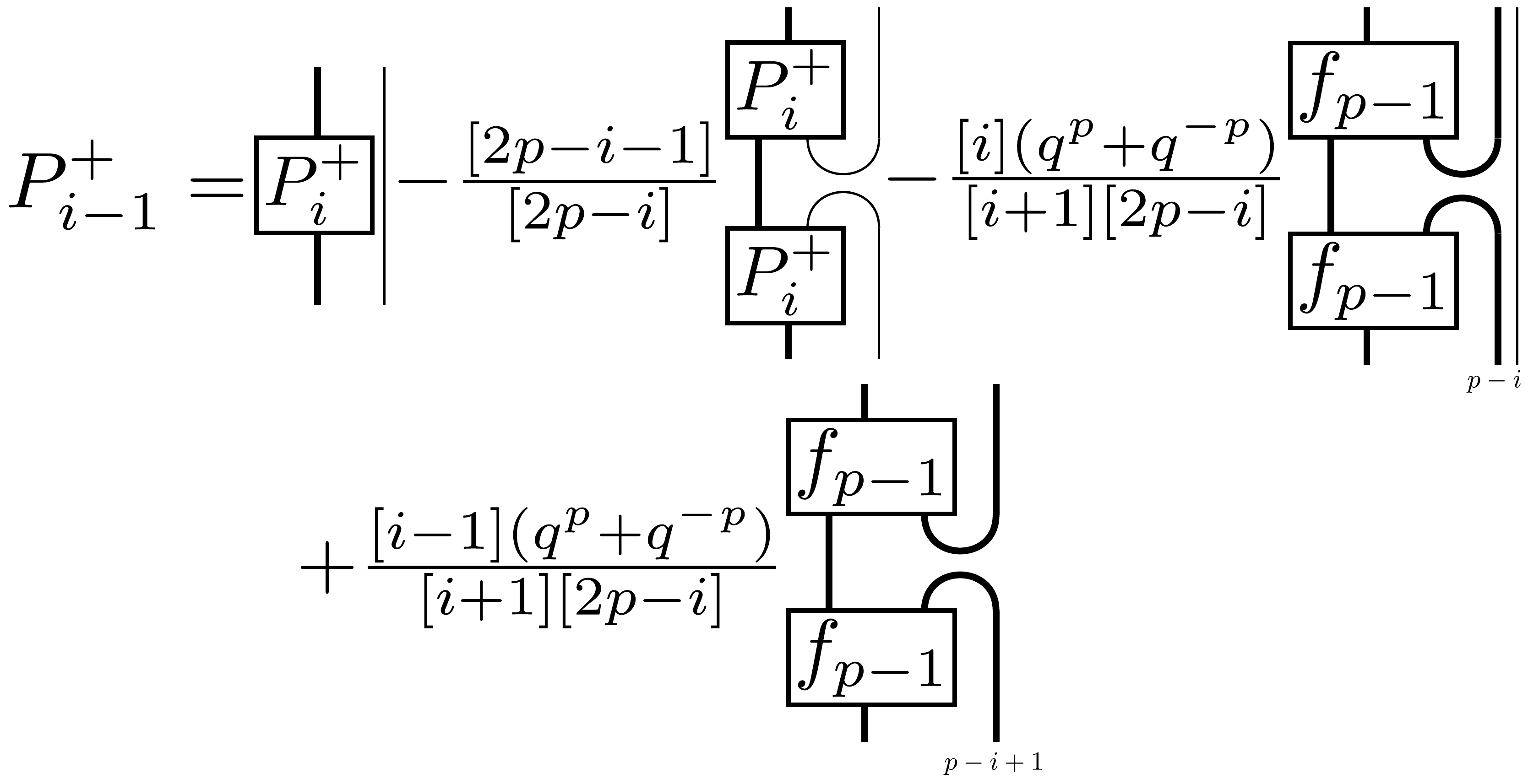}
\end{figure}
Next we want to state the following basic properties of the projections:
\begin{lemma}
	The indecomposable projections satisfy the following:
	\begin{figure}[H]
		\centering
		\includegraphics[width=0.3\linewidth]{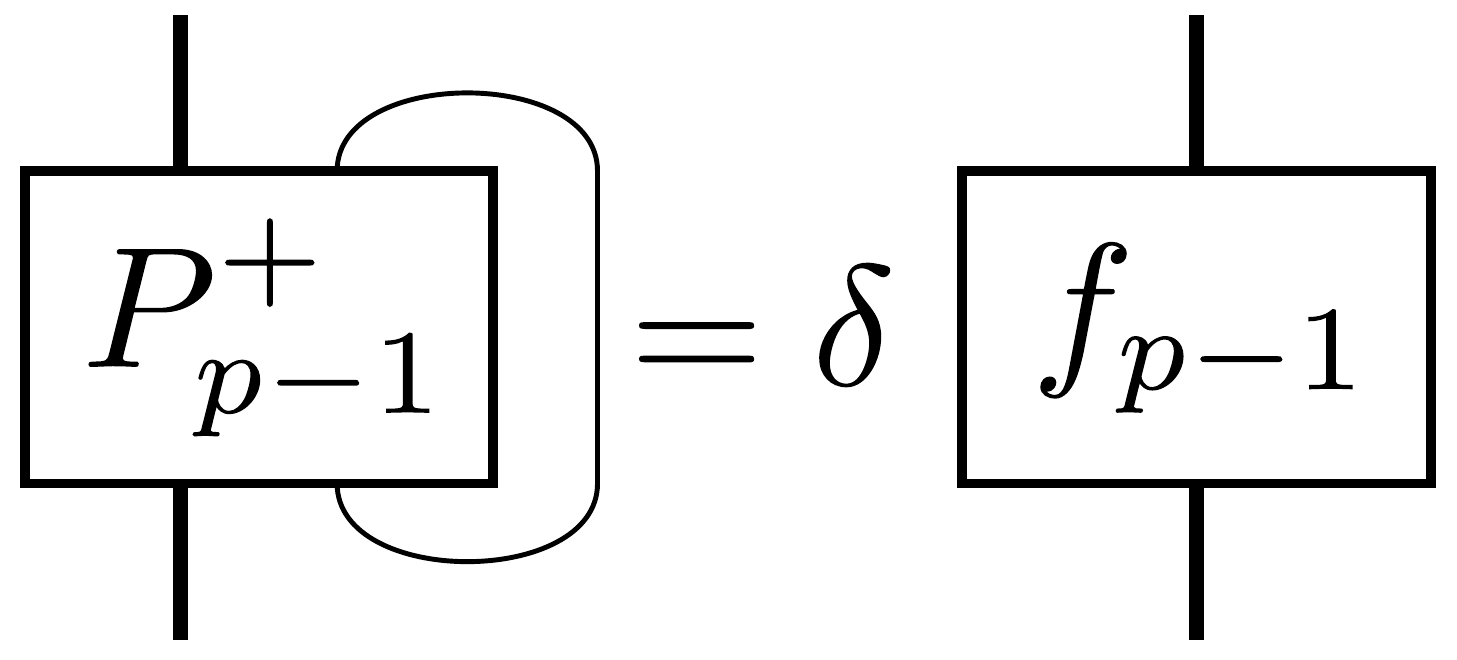}
	\end{figure}
	\begin{figure}[H]
		\centering
		\includegraphics[width=0.9\linewidth]{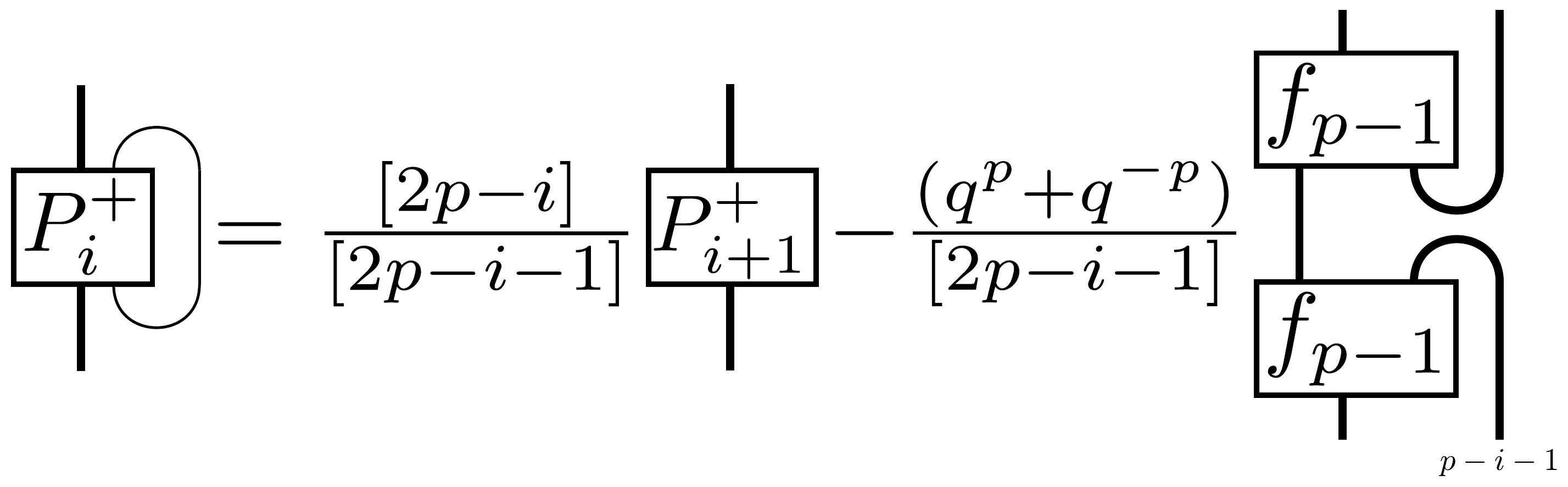}
	\end{figure}
	\begin{figure}[H]
		\centering
		\includegraphics[width=1\linewidth]{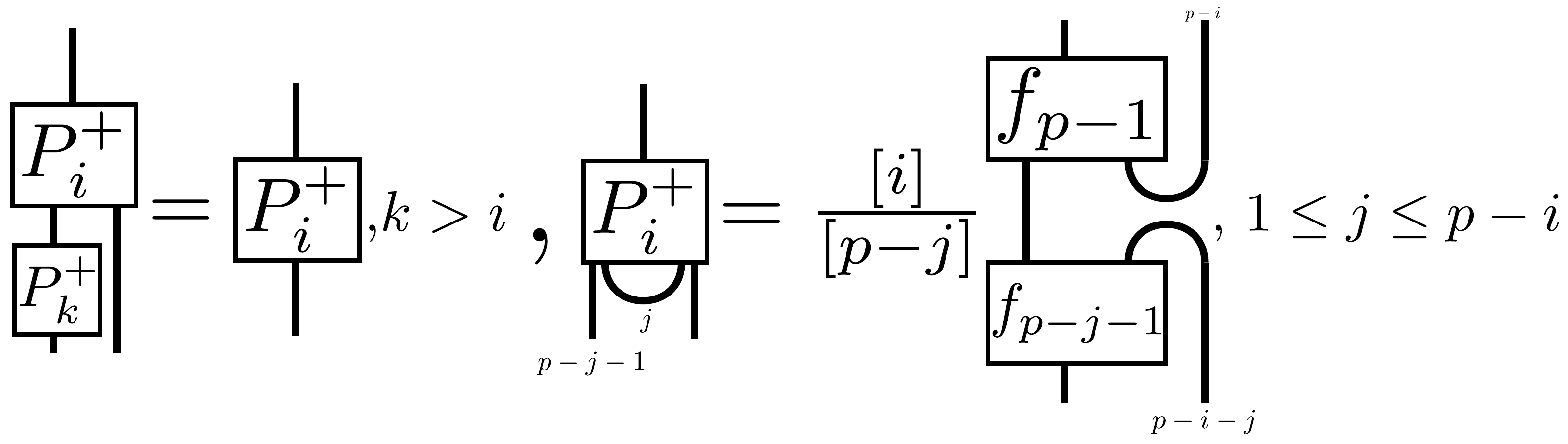}
	\end{figure}

	along with their reflections about the horizontal axis. 
\end{lemma}
We note that in general, left and right partial traces of the indecomposable projections are not equal.
Given these relations, we can now proceed to prove Theorem \ref{ind proj}. 
\begin{proof}
	The previous lemma shows that our formulae is finite when evaluated for our choice of $q$. That the formulae are projections is straightforward to show, and is given in Section $6$ of \cite{methesis}. To show that they are the correct projections, we give explicit formula for isomorphism maps for the appropriate fusion rules in the $\bar{U}_{q}(\mathfrak{sl}_{2})$ representation category, i.e. for an isomorphism $A\simeq B$, we have a map $V$ such that $V^{T}V$ is the identity on $A$, and $VV^{T}$ is the identity on $B$, where $V^{T}$ denotes the transpose of the vector $V$, with diagram elements reflected about the horizontal axis.\\
	The fusion rule $\mathcal{X}^{+}_{p}\otimes X\simeq \mathcal{P}^{+}_{p-1}$ is immediate.\\
	For $p=2$, the fusion rule $\mathcal{P}^{+}_{1}\otimes X\simeq 2\mathcal{P}^{+}_{2}\oplus 2\mathcal{P}^{-}_{2}$ has the following isomorphism map:
	\begin{figure}[H]
		\centering
		\includegraphics[width=0.2\linewidth]{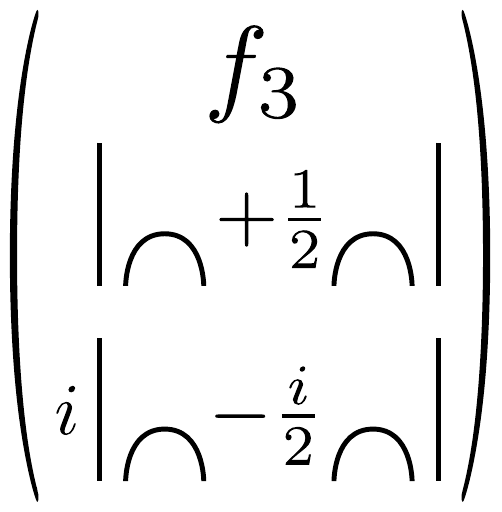}
	\end{figure}
	For $p>2$, we have the following fusion rules and isomorphism maps:
	$\mathcal{P}^{+}_{p-1}\otimes X\simeq\mathcal{P}^{+}_{p-2}\oplus 2\mathcal{P}^{+}_{p}$ has isomorphism map:
	\begin{figure}[H]
		\centering
		\includegraphics[width=0.3\linewidth]{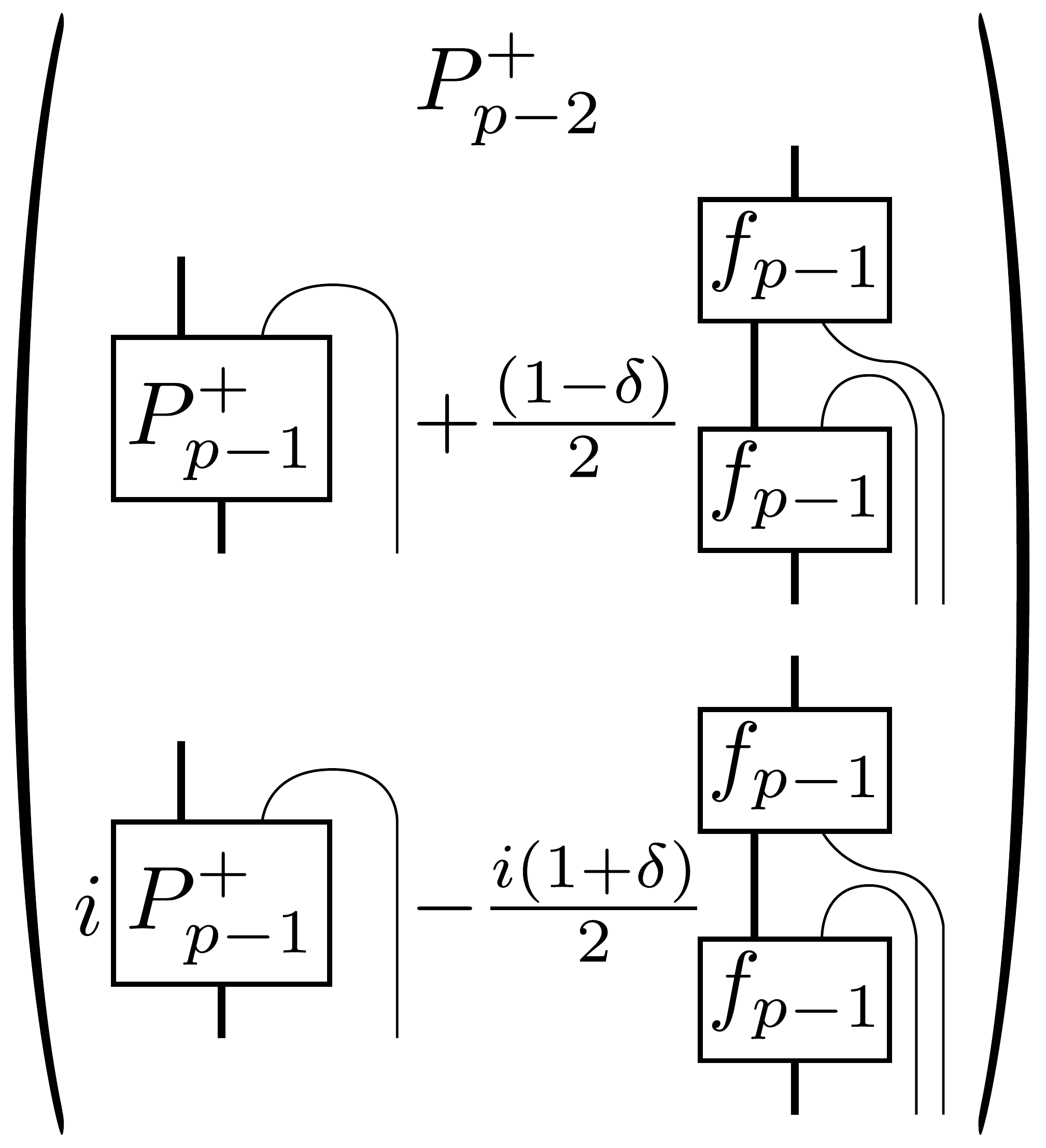}
	\end{figure}
	The fusion rule $\mathcal{P}^{+}_{i}\otimes X\simeq\mathcal{P}^{+}_{i-1}\oplus\mathcal{P}^{+}_{i+1}$, $2\leq i\leq p-2$, has isomorphism map:
	\begin{figure}[H]
		\centering
		\includegraphics[width=0.5\linewidth]{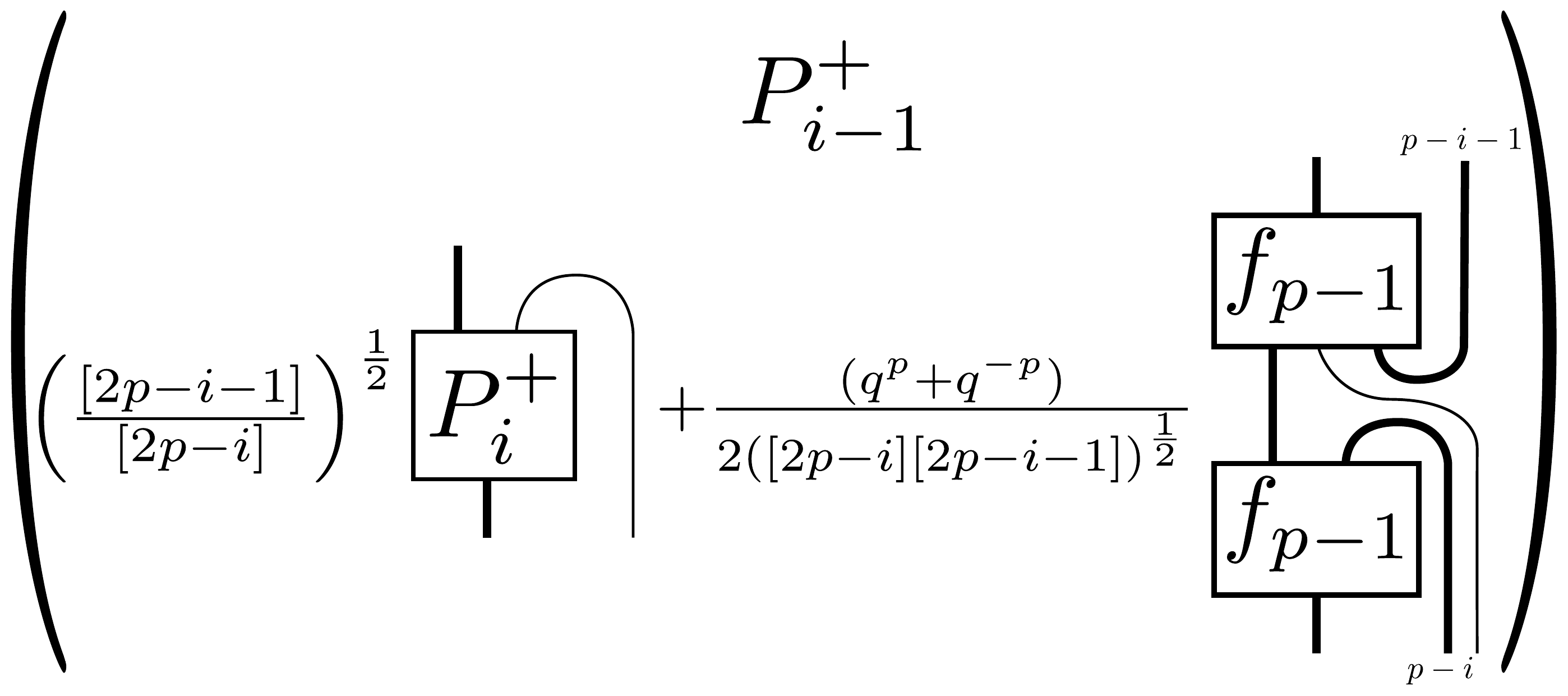}
	\end{figure}
	The fusion rule $\mathcal{P}^{+}_{1}\otimes X\simeq \mathcal{P}^{+}_{2}\oplus 2\mathcal{P}^{-}_{p}$ has isomorphism map:
	\begin{figure}[H]
		\centering
		\includegraphics[width=0.5\linewidth]{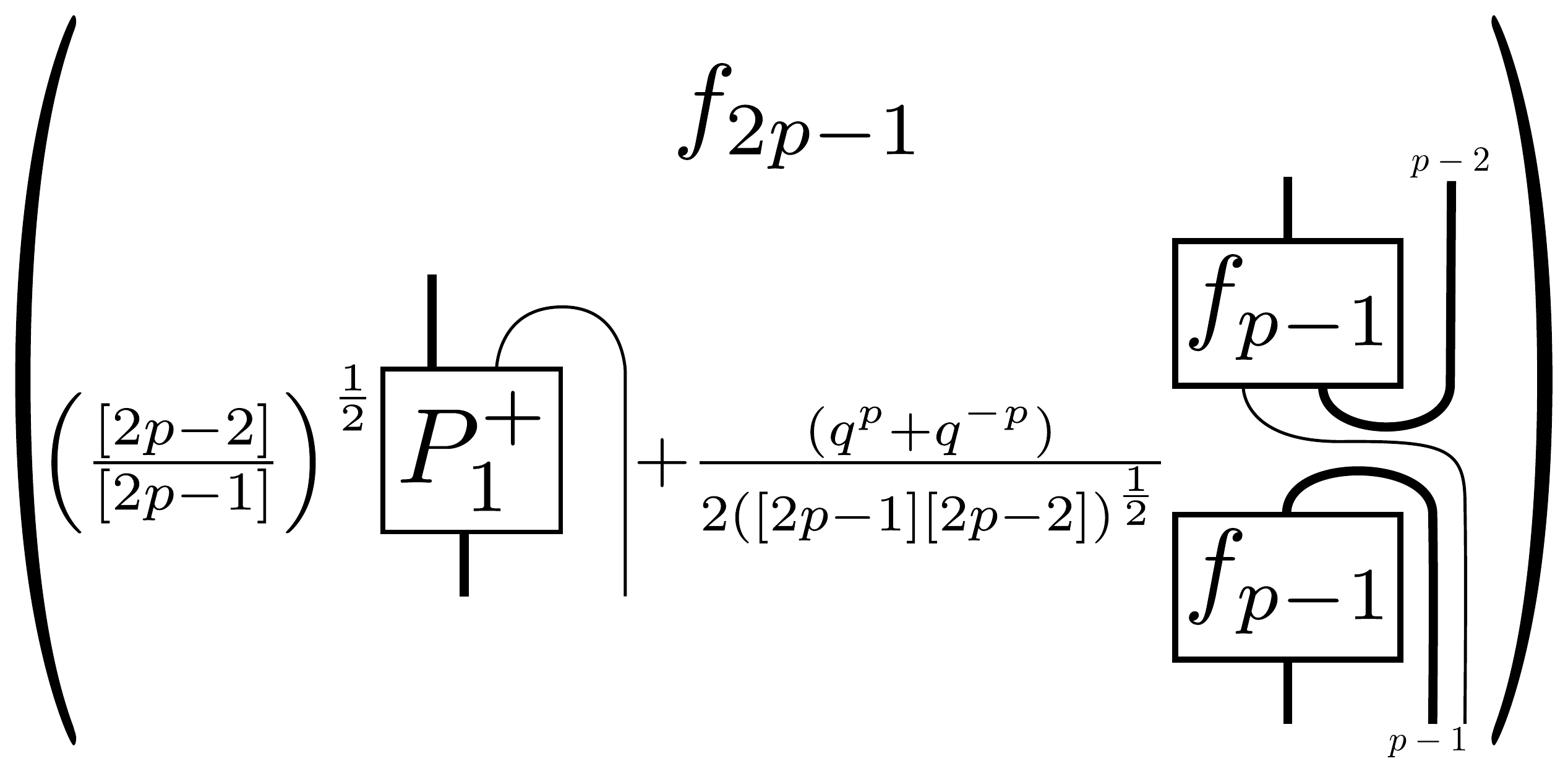}
	\end{figure}
	The $p=2$ case, and next two isomorphism maps are straightforward to verify, based on the results about the inductive formula for the projections and their partial trace given previously. For the last map, we have:
	\begin{figure}[H]
		\centering
		\includegraphics[width=0.5\linewidth]{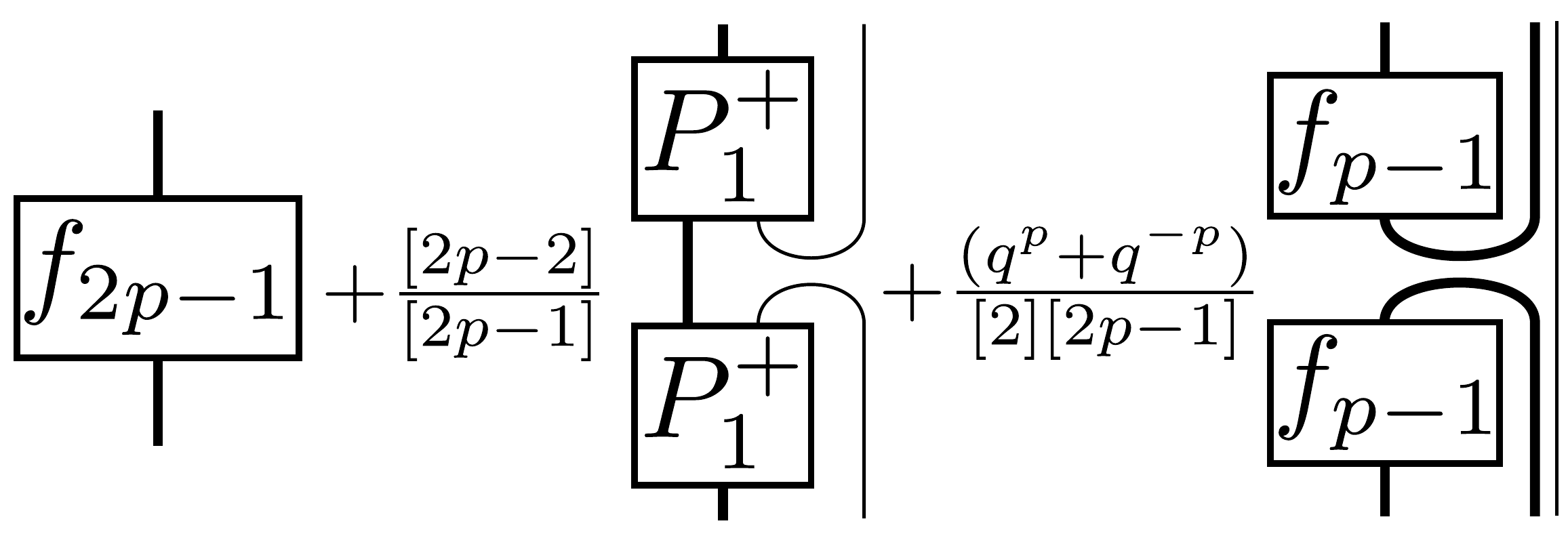}
	\end{figure}
	\begin{figure}[H]
		\centering
		\includegraphics[width=1\linewidth]{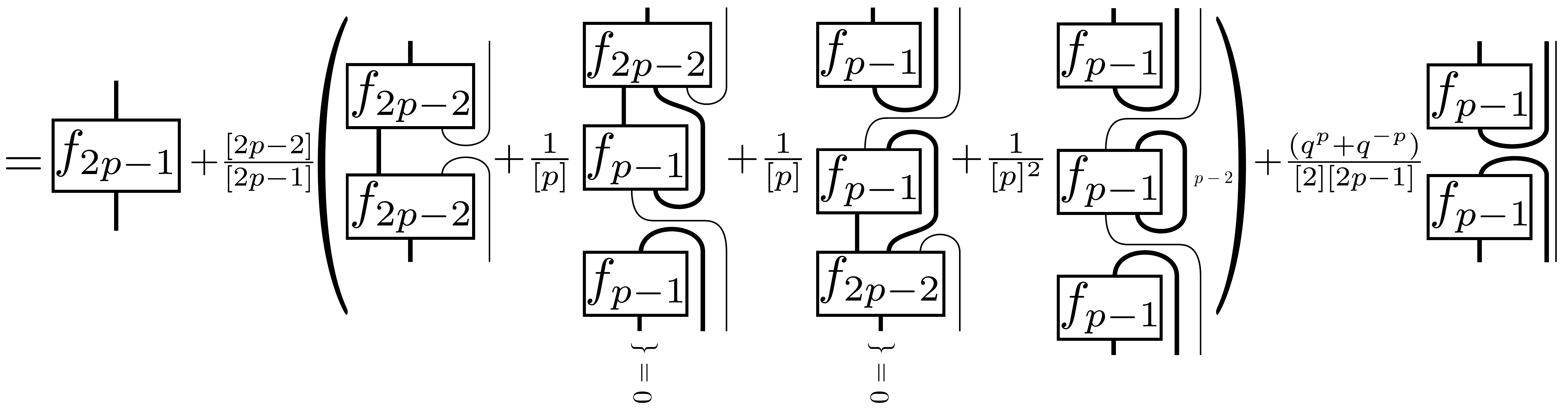}
	\end{figure}
	\begin{figure}[H]
		\centering
		\includegraphics[width=0.9\linewidth]{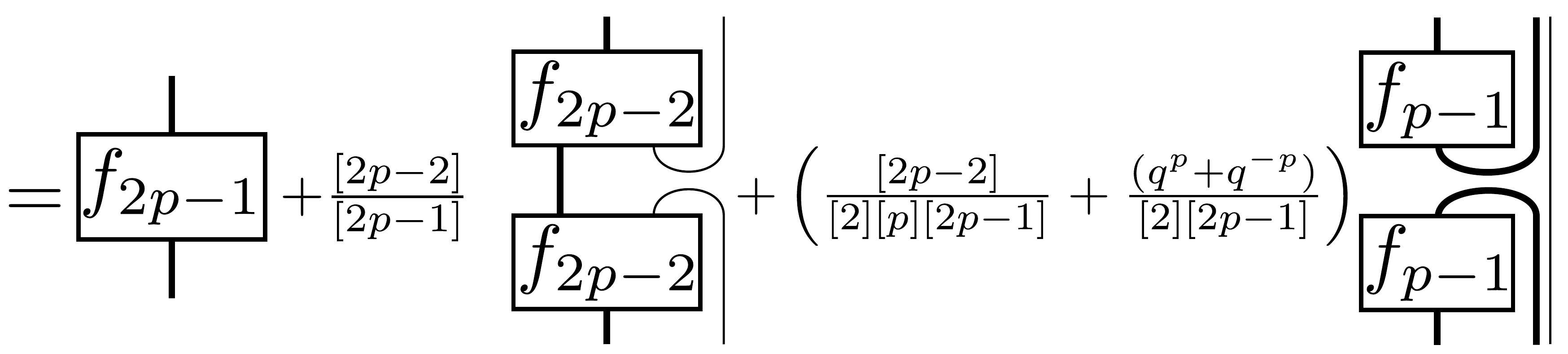}
	\end{figure}
	\begin{figure}[H]
		\centering
		\includegraphics[width=0.4\linewidth]{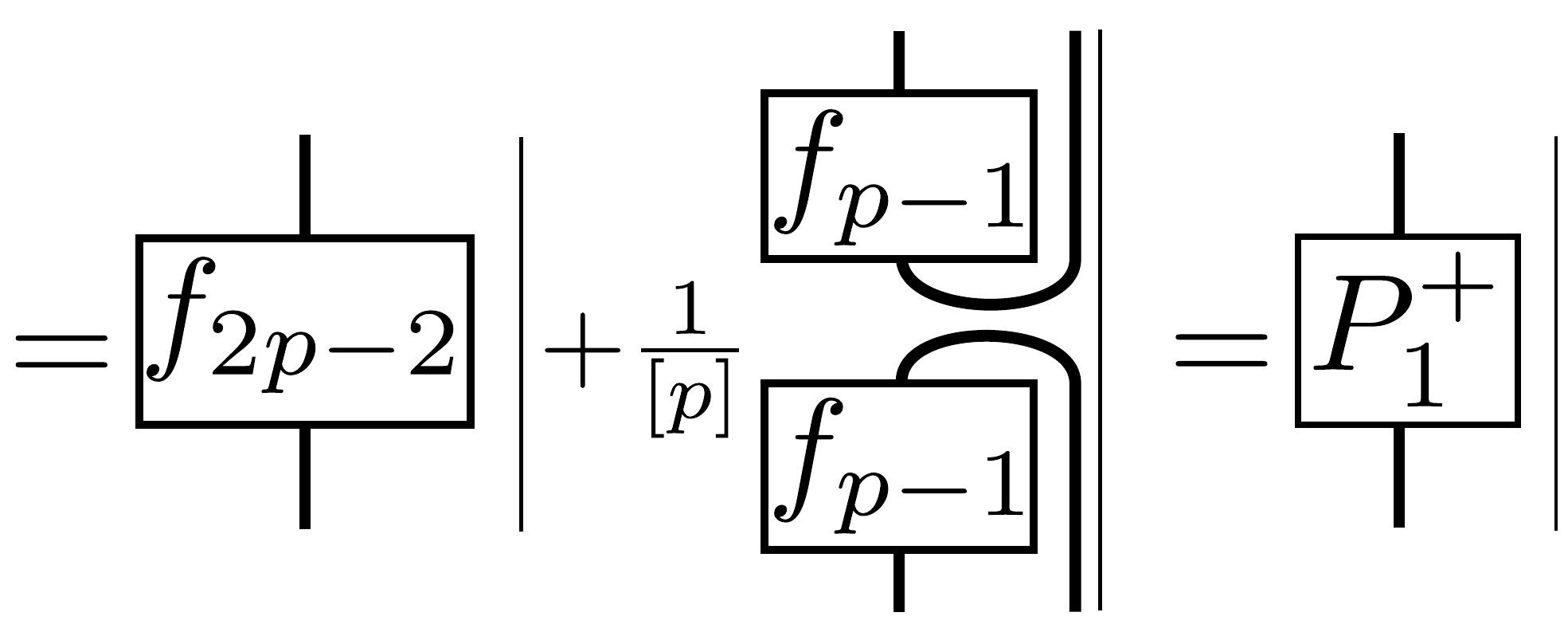}
	\end{figure}
	where for the last simplification, we used the following:
	\begin{align*}
	\frac{[2p-2]+[p](q^{p}+q^{(-p)})}{[2][2p-1]}&=\frac{q^{(2p-2)}-q^{(2-2p)}+q^{2p}-q^{(-2p)}}{q^{2p}+q^{2p-2}-q^{2-2p}-q^{-2p}}=1
	\end{align*}

\end{proof}

\begin{remark}
The projections and isomorphism maps depend on a choice made about the fusion rule $\mathcal{X}^{+}_{p}\otimes X$, as we have $\mathcal{X}^{+}_{p}\otimes X\simeq X\otimes\mathcal{X}^{+}_{p}$. We further chose to continue tensoring by $X$ on the right, i.e. $((((\mathcal{X}^{+}_{p}\otimes X)\otimes X)...)\otimes X)$, however we can also choose to tensor by $X$ on the left or some mix of left and right. In these cases, the appropriate diagram for the projection is given by replacing the second diagram in the original projection formula with one such that a number of the cupped strings are to the left. The proof still holds in these cases as the Jones-Wenzl projections are symmetric under reflection about the vertical axis. It turns out that even sums of these projections can themselves be projections. This can be thought of as being due to a non-unique choice for the assignments $X^{\otimes 2p-i-1}\rightarrow a_{k}$ and $b_{k}\rightarrow X^{\otimes 2p-i-1}$ in the map $X^{\otimes 2p-i-1}\rightarrow \mathcal{P}^{+}_{i}\rightarrow X^{2p-i-1}$. 
\end{remark}

\section{Other Morphisms between Indecomposable Modules}\label{morphisms}
The remainder of this paper is devoted to the non-identity maps on indecomposable modules. Namely, the maps between $\mathcal{P}^{+}_{i}$ and $\mathcal{X}^{+}_{i}$, the maps between $\mathcal{P}^{+}_{i}$ and $\mathcal{P}^{-}_{p-i}$, and the second endomorphisms on $\mathcal{P}^{\pm}_{i}$. We begin with the following:
\begin{prop}
Up to a constant, the homomorphism \begin{align*}
\theta:X^{\otimes 2p-i-1}\rightarrow \mathcal{P}^{+}_{i}\rightarrow \mathcal{X}^{+}_{i}\rightarrow X^{\otimes i-1}
\end{align*} is given diagrammatically by:
\begin{figure}[H]
	\centering
	\includegraphics[width=0.25\linewidth]{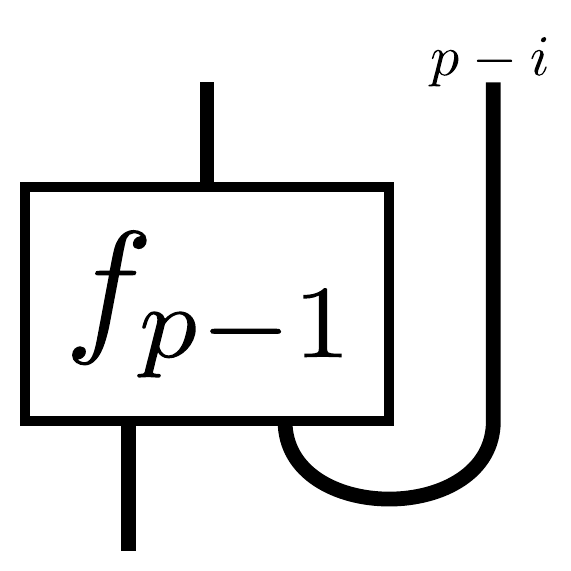}
\end{figure}
\end{prop}
\begin{proof}
In terms of basis elements, using Sections $4.3$ and $6.2$ of \cite{methesis}, the homomorphism is a composition of the following maps:
\begin{align*}
X^{\otimes 2p-i-1}\rightarrow \mathcal{P}^{+}_{i}:& \rho_{i_{1},...,i_{n},2p-i-1}\mapsto (-1)^{p-i-1}q^{\big(n(2p-i-1)-\frac{1}{2}(n^{2}-n)-(\sum\limits_{j=1}^{n}i_{j})\big)}\times\\
&\times\frac{([n]!)([p-n-1]!)}{([n+i-p]!)([i-1]!)([p-i-1]!)^{2}}b_{n+i-p}, \:\: p-i\leq n\leq p-1\\
\mathcal{P}^{+}_{i}\rightarrow \mathcal{X}^{+}_{i}:& x_{j}\mapsto 0, \:\:\: a_{k}\mapsto 0, \:\:\: y_{j}\mapsto 0, \:\:\: b_{k}\mapsto z_{k}, \:\: 0\leq k\leq i-1\\
\mathcal{X}^{+}_{i}\rightarrow X^{\otimes i-1}:& z_{k}\mapsto F^{k}x_{0,i-1}
\end{align*}
Combining the maps, we get:
\begin{align*}
\theta(\rho_{i_{1},...,i_{n},2p-i-1})=& (-1)^{p-i-1}q^{\big(n(2p-i-1)-\frac{1}{2}(n^{2}-n)-(\sum\limits_{j=1}^{n}i_{j})\big)}\times\\
&\times\frac{([n]!)([p-n-1]!)}{([n+i-p]!)([i-1]!)([p-i-1]!)^{2}}F^{n+i-p}x_{0,i-1}\\
& p-i\leq n\leq p-1\\
\theta(\rho_{i_{1},...,i_{n},2p-i-1})=& 0, \:\: 0\leq n< p-i, \:\: p-1< n\leq 2p-i-1
\end{align*}
	Given an element $\rho_{i_{1},...,i_{n},2p-i-1}$, we can rewrite it as:
	\begin{align*}
	\rho_{i_{1},...,i_{r},p-1}\otimes\rho_{(i_{r+1}+1-p),...,(i_{n}+1-p),p-i}\in X_{n,2p-i-1}
	\end{align*}   Acting $f_{p-1}\otimes 1^{\otimes p-i}$ on this we get:
	\begin{align*}
	& \sum\limits_{m=0}^{r}q^{\big(r(p-1)-\frac{1}{2}(r^{2}-r)+2m(r-m)-m(i-1)-(\sum\limits_{j=1}^{r}i_{j})\big)}\frac{([p-1-r]!)}{([p-1]!)}\lambda_{m,r}\times\\
	&\times(F^{r-m}x_{0,i-1})\otimes(F^{m}x_{0,p-i})\otimes \rho_{(i_{r+1}+1-p),...,(i_{n}+1-p),p-i}
	\end{align*}
	We now want to apply $(p-i)$ copies of $\cup$ to this. Remembering our definition of the Temperley-Lieb action, we have $\cup(\nu_{00})=\cup(\nu_{11})=0$, $\cup(\nu_{10})=1$, $\cup(\nu_{01})=-q$. This means that the number of zeros in $(F^{m}x_{0,p-i})$ must be equal to the number of ones in $\rho_{(i_{r+1}+1-p),...,(i_{n}+1-p),p-i}$. Hence we need $p-i-m=n-r$, and so $m=p-i-n+r$. (Note that as $m\leq r$, this requires $n\geq p-i$ to be non-zero). This then gives:
	\begin{align*}
	& \sum\limits_{1\leq k_{l}\leq p-i}\frac{([p-1-r]!)([p-i-n+r]!)}{([p-1]!)}\lambda_{p-i-n+r,r}\times\\
	&\times q^{\big(r(p-1)-\frac{1}{2}(r^{2}-r)+(p-i-n+r)(2n+i-2p+1)+\frac{1}{2}(p-i-n+r)(p-i-n+r+1)-(\sum\limits_{j=1}^{r}i_{j})-(\sum\limits_{l=1}^{p-i-n+r}k_{l})\big)}\times\\
	&\times(F^{n+i-p}x_{0,i-1})\otimes\rho_{k_{1},...,k_{p-i-n+r},p-i}\otimes \rho_{(i_{r+1}+1-p),...,(i_{n}+1-p),p-i}
	\end{align*}
	Denote the positions of the zeros in $\rho_{k_{1},...,k_{p-i-n+r},p-i}$ by $\tilde{k}_{1},...,\tilde{k}_{n-r}$. To apply $\cup$ $(p-i)$ times to $\rho_{k_{1},...,k_{p-i-n+r},p-i}\otimes\rho_{(i_{r+1}+1-p),...,(i_{n}+1-p),p-i}$, we need that\\ $(p-i+1-\tilde{k}_{n-r})=(i_{r+1}+1-p),...,(p-i+1-\tilde{k}_{1})=(i_{n}+1-p)$, and so\\ $(n-r)(1-p)+\sum\limits_{j=r+1}^{n}i_{j}=(n-r)(p-i+1)-(\sum\limits_{l=1}^{n-r}\tilde{k}_{l})$. However, we also have\\ $\sum\limits_{l=1}^{p-i-n+r}k_{l}+\sum\limits_{m=1}^{n-r}\tilde{k}_{m}=\sum\limits_{k=1}^{p-i}k=\frac{1}{2}(p-i)(p-i+1)$. Hence we have\\ $\sum\limits_{l=1}^{p-i-n+r}k_{l}=\frac{1}{2}(p-i)(p-i+1)+(n-r)(i-2p)+\sum\limits_{j=r+1}^{n}i_{j}$.\\ 
	We can now apply $\cup$ $(p-i)$ times to get:
	\begin{align*}
	& (-1)^{p-i-1}q^{\big(n(2p-i-1)-\frac{1}{2}(n^{2}-n)-i-(\sum\limits_{j=1}^{n}i_{j})\big)}\frac{([p-1-r]!)([r]!)}{([p-1]!)([n+i-p]!)}F^{n+i-p}x_{0,i-1}
	\end{align*}
	Note that as $[p-j]=[j]$, we have $([r]!)([p-r-1]!)=[p-1]!=([n]!)([p-n-1]!)$. Hence this is equal to:
	\begin{align*}
	-q^{-i}\frac{([p-i-1]!)}{[i]}\theta(\rho_{i_{1},,...,i_{n},2p-i-1})
	\end{align*}
	Note that we included an extra minus sign to account for our definition of $\cup$. Note also, that as we have $F^{n+i-p}x_{0,i-1}$, $n=p$ gives $F^{i}x_{0,i-1}=0$, and so we must have $n\leq p-1$ for this to be non-zero.  
\end{proof}

\begin{prop}
Up to a constant, the homomorphism \begin{align*}
\Gamma: X^{\otimes i-1}\rightarrow\mathcal{X}^{+}_{i}\rightarrow\mathcal{P}^{+}_{i}\rightarrow X^{\otimes 2p-i-1}
\end{align*} is given diagrammatically by:	
	\begin{figure}[H]
		\centering
		\includegraphics[width=0.25\linewidth]{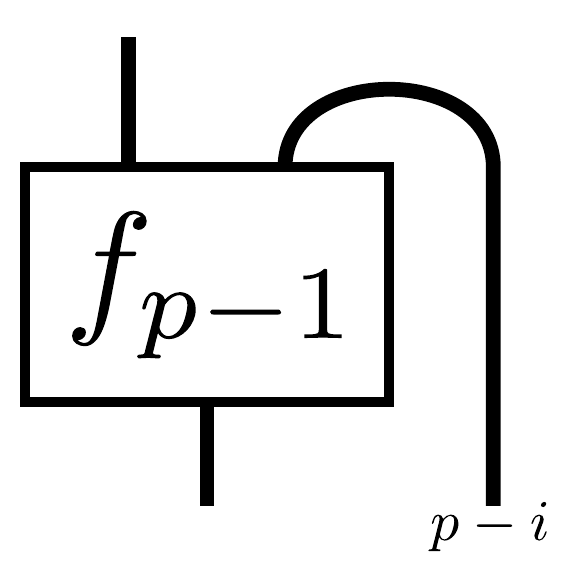}
	\end{figure}
\end{prop}
\begin{proof}
In terms of basis elements this is a composition of the following maps:
\begin{align*}
X^{\otimes i-1}\rightarrow\mathcal{X}^{+}_{i}:&\rho_{i_{1},...i_{n},i-1}\mapsto q^{\big(n(i-1)-\frac{1}{2}(n^{2}-n)-(\sum\limits_{j=1}^{n}i_{j})\big)}\frac{([i-1-n]!)}{([i-1]!)}z_{n}, \:\: 0\leq n\leq i-1\\
\mathcal{X}^{+}_{i}\rightarrow\mathcal{P}^{+}_{i}:& z_{n}\mapsto a_{n}\\
\mathcal{P}^{+}_{i}\rightarrow X^{\otimes 2p-i-1}:& a_{n}\mapsto F^{n+p-i}x_{0,2p-i-1}
\end{align*}
Taking their composition we get:
\begin{align*}
\Gamma(\rho_{i_{1},...,i_{n},i-1})=& q^{\big(n(i-1)-\frac{1}{2}(n^{2}-n)-(\sum\limits_{j=1}^{n}i_{j})\big)}\frac{([i-1-n]!)}{([i-1]!)}F^{n+p-i}x_{0,2p-i-1} \:\: 0\leq n\leq i-1
\end{align*}
	Given $\rho_{j_{1},...,j_{k},z}$, let $\tilde{j}_{1},...,\tilde{j}_{z-k}$ be the positions of the zeros. As $\cap(\nu)=q^{-1}\nu_{10}-\nu_{01}$, we have that the $z$-fold cap is given by:
	\begin{align*}
	\sum\limits_{k=0}^{z}\bigg(\sum\limits_{1\leq j_{l}\leq z}(-1)^{z-k}q^{-k}\rho_{j_{1},...,j_{k},(2z+1-\tilde{j}_{z-k}),...,(2z+1-\tilde{j}_{1}),2z}\bigg)
	\end{align*}
	Hence given $\rho_{i_{1},...,i_{n},i-1}$, we want to consider the action of $(f_{p-1}\otimes 1^{\otimes p-i})$ on the following:
	\begin{align*}
	\sum\limits_{k=0}^{p-i}\bigg(\sum\limits_{1\leq j_{l}\leq p-i}(-1)^{p-i-k}q^{-k}\rho_{i_{1},...,i_{n},i-1}\otimes\rho_{j_{1},...,j_{k},(2p-2i+1-\tilde{j}_{p-i-k}),...,(2p-2i+1-\tilde{j}_{1}),2p-2i}\bigg)
	\end{align*}
	This is given by:
	\begin{align*}
	\sum\limits_{k=0}^{p-i}&\bigg(\sum\limits_{1\leq j_{l}\leq p-i}(-1)^{p-i-k}q^{\big((n+k)(p-1)-\frac{1}{2}(n+k)(n+k-1)-ki-(\sum\limits_{j=1}^{n}i_{j})-(\sum\limits_{l=1}^{k}j_{l})\big)}\times\\
	&\times\frac{([p-1-n-k]!)}{([p-1]!)}(F^{n+k}x_{0,p-1})\otimes\rho_{(p-i+1-\tilde{j}_{p-i-k}),...,(p-i+1-\tilde{j}_{1}),p-i}\bigg)
	\end{align*}
	Note that as $\sum\limits_{l=1}^{k}j_{l}+\sum\limits_{m=1}^{p-i-k}\tilde{j}_{m}=\frac{1}{2}(p-i)(p-i+1)$, we have:
	\begin{align*}
	\sum\limits_{k=0}^{p-i}&\bigg(\sum\limits_{1\leq j_{l}\leq p-i}(-1)^{p-i-k}q^{\big((n+k)(p-1)-\frac{1}{2}(n+k)(n+k-1)-ki-(\sum\limits_{j=1}^{n}i_{j})-\frac{1}{2}(p-i)(p-i+1)+(\sum\limits_{l=1}^{p-i-k}\tilde{j}_{l})\big)}\times\\
	&\times\frac{([p-1-n-k]!)}{([p-1]!)}(F^{n+k}x_{0,p-1})\otimes\rho_{(p-i+1-\tilde{j}_{p-i-k}),...,(p-i+1-\tilde{j}_{1}),p-i}\bigg)\\
	=\sum\limits_{k=0}^{p-i}&\bigg(\sum\limits_{1\leq m_{r}\leq p-1}\sum\limits_{1\leq j_{l}\leq p-i}q^{\big(p(n+k)-ki-(\sum\limits_{j=1}^{n}i_{j})-\frac{1}{2}(p-i)(p-i+1)+(\sum\limits_{l=1}^{p-i-k}\tilde{j}_{l})-(\sum\limits_{r=1}^{n+k}m_{r})\big)}\times\\
	&\times(-1)^{p-i-k}\frac{([p-1-n-k]!)([n+k]!)}{([p-1]!)}\rho_{m_{1},...,m_{n+k},(2p-i-\tilde{j}_{p-i-k}),...,(2p-i-\tilde{j}_{1}),2p-i-1}\bigg)
	\end{align*}
	Let $m_{n+k+1}:=2p-i-\tilde{j}_{p-i-k}$,...,$m_{p-i+n}:=2p-i-\tilde{j}_{1}$, then \\ $\sum\limits_{l=n+k+1}^{p-i+n}m_{l}=(p-i-k)(2p-i)-(\sum\limits_{m=1}^{p-i-k}\tilde{j}_{m})$, and we can rewrite as:
	\begin{align*}
	\sum\limits_{k=0}^{p-i}&\bigg(\sum\limits_{1\leq m_{j}\leq 2p-i-1}q^{\big(np-(\sum\limits_{j=1}^{n}i_{j})-\frac{1}{2}(p-i)(p-i+1)+(p-i)(2p-i)-(\sum\limits_{j=1}^{n+p-i}m_{j})\big)}\times\\
	&\times(-1)^{p-i}\rho_{m_{1},...,m_{n+k},r_{n+k+1},...,m_{n+p-i},2p-i-1}\bigg)\\
	=& \sum\limits_{1\leq m_{j}\leq 2p-i-1}(-1)^{p-i}q^{\big(np-(\sum\limits_{j=1}^{n}i_{j})-\frac{1}{2}(p-i)(p-i+1)+(p-i)(2p-i)-(\sum\limits_{j=1}^{n+p-i}m_{j})\big)}\rho_{m_{1},...m_{n+p-i},2p-i-1}\\
	=& \frac{q^{\big(i(n+1)-p-\frac{1}{2}(n^{2}+n)-(\sum\limits_{j=1}^{n}i_{j})\big)}}{([n+p-i]!)}F^{n+p-i}x_{0,2p-i-1}\\
	=& q^{i}\frac{([i-1]!)}{([p-1]!)}\Gamma(\rho_{i_{1},...,i_{n},i-1})
	\end{align*}
	where again we included an extra minus sign to account for our choice of $\cap$.
\end{proof}

\begin{corr}
The second endomorphism on $\mathcal{P}^{+}_{i}$, is given diagrammatically by:
\begin{figure}[H]
	\centering
	\includegraphics[width=0.25\linewidth]{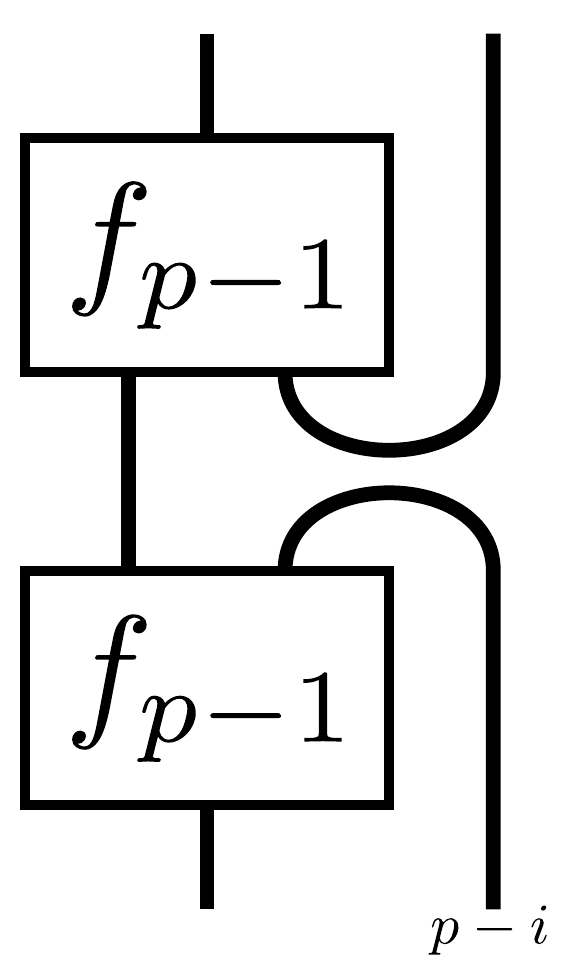}
\end{figure}
\end{corr}
\begin{proof}
Taking the composition of the maps of the previous two sections, we have:
\begin{align*}
\Phi=\Gamma\theta:X^{\otimes 2p-i-1}\rightarrow \mathcal{P}^{+}_{i}\rightarrow\mathcal{X}^{+}_{i}\rightarrow X^{\otimes i-1}\rightarrow \mathcal{X}^{+}_{i}\rightarrow \mathcal{P}^{+}_{i}\rightarrow X^{\otimes 2p-i-1}
\end{align*}
It is straightforward to show that the map $\mathcal{X}^{+}_{i}\rightarrow X^{\otimes i-1}\rightarrow \mathcal{X}^{+}_{i}$ is equal to the identity. It then follows that the map $\mathcal{P}^{+}_{i}\rightarrow\mathcal{X}^{+}_{i}\rightarrow\mathcal{P}^{+}_{i}$ is equal to the second endomorphism on $\mathcal{P}^{+}_{i}$. Hence the composition of the two maps is equal to the second endomorphism.
\end{proof}
\begin{remark}
We noted at the end of the previous section that we could define alternate projections by replacing the above diagram with one where some of the strings are on the left hand side instead. We can similarly replace the second endomorphism with these alternate diagrams as well. However at the appropriate value of $\delta$, these are all equal.
\end{remark}

\subsection{The homomorphisms $\mathcal{P}^{\pm}_{i}\leftrightarrow\mathcal{P}^{\mp}_{p-i}$}

The homomorphisms $\theta:\mathcal{P}^{+}_{i}\rightarrow\mathcal{P}^{-}_{p-i}$, $\Gamma:\mathcal{P}^{-}_{p-i}\rightarrow\mathcal{P}^{+}_{i}$, are given as follows:
\begin{align*}
\theta(a_{m})&=0 & \theta(b_{m})&=g_{1}\tilde{x}_{m}+g_{2}\tilde{y}_{m}& \theta(x_{n})&=g_{2}\tilde{a}_{n}& \theta(y_{n})&=g_{1}\tilde{a}_{n}\\
\Gamma(\tilde{a}_{n})&=0 & \Gamma(\tilde{b}_{n})&=k_{1}x_{n}+k_{2}y_{n} & \Gamma(\tilde{x}_{m})&=k_{2}a_{m} & \Gamma(\tilde{y}_{m})&=k_{1}a_{m}
\end{align*}
for $0\leq m\leq i-1$, $0\leq n\leq p-i-1$, and $g_{1},g_{2},k_{1},k_{2}\in\mathbb{K}$. We denote the case $g_{1}=1,g_{2}=0$ by $\theta_{1}$, $g_{1}=0,g_{2}=1$ by $\theta_{2}$, $k_{1}=1,k_{2}=0$ by $\Gamma_{1}$, and $k_{1}=0,k_{2}=1$ by $\Gamma_{2}$.

We want to describe these homomorphisms diagrammatically. For $\theta$, this will be a diagram given by \begin{align*}
\tilde{\theta}_{j}&:X^{\otimes 2p-1-i}\rightarrow \mathcal{P}^{+}_{i}\rightarrow\mathcal{P}^{-}_{p-i}\rightarrow X^{\otimes (2p-1+i)}
\end{align*} 
Note that there are two copies of $\mathcal{P}^{-}_{p-i}$ appearing in $X^{\otimes 2p-1+i}$, one in the weight spaces $X_{0,2p-1+i},\\...,X_{p+i,2p-1+i}$, the other in the weight spaces $X_{p-1,2p-1+i},...,X_{2p-1+i,2p-1+i}$. Denote the maps onto these as $\theta_{j,l}$, $\theta_{j,u}$ respectively, with $j\in\{1,2\}$. In terms of the weight spaces, these maps can be characterized as:
\begin{align*}
\theta_{1,l}&:X_{k,2p-i-1}\rightarrow X_{k+i-p,2p+i-1} & p-i\leq & k\leq 2p-i-1\\
\theta_{2,l}&:X_{k,2p-i-1}\rightarrow X_{k+i,2p+i-1} & 0\leq & k\leq p-1\\
\theta_{1,u}&:X_{k,2p-i-1}\rightarrow X_{k+i,2p+i-1} & p-i\leq & k\leq 2p-i-1\\
\theta_{2,u}&:X_{k,2p-i-1}\rightarrow X_{k+p+i,2p+i-1} & 0\leq & k\leq p-1
\end{align*}
Similarly, $\Gamma$ can be characterized as:
\begin{align*}
\Gamma_{1,l}&:X_{k,2p+i-1}\rightarrow X_{k-i,2p-i-1} & i\leq & k\leq p+i-1\\
\Gamma_{2,l}&:X_{k,2p+i-1}\rightarrow X_{k+p-i,2p-i-1} & 0\leq & k\leq p-1\\
\Gamma_{1,u}&:X_{k,2p+i-1}\rightarrow X_{k-p+i,2p-i-1} & p+i\leq & k\leq 2p+i-1\\
\Gamma_{2,u}&:X_{k,2p+i-1}\rightarrow X_{k-i,2p-i-1} & p\leq & k\leq 2p-1
\end{align*}
\begin{prop}
Diagrammatically, the maps $\theta,\Gamma$ are given by:
\begin{figure}[H]
	\centering
	\includegraphics[width=1\linewidth]{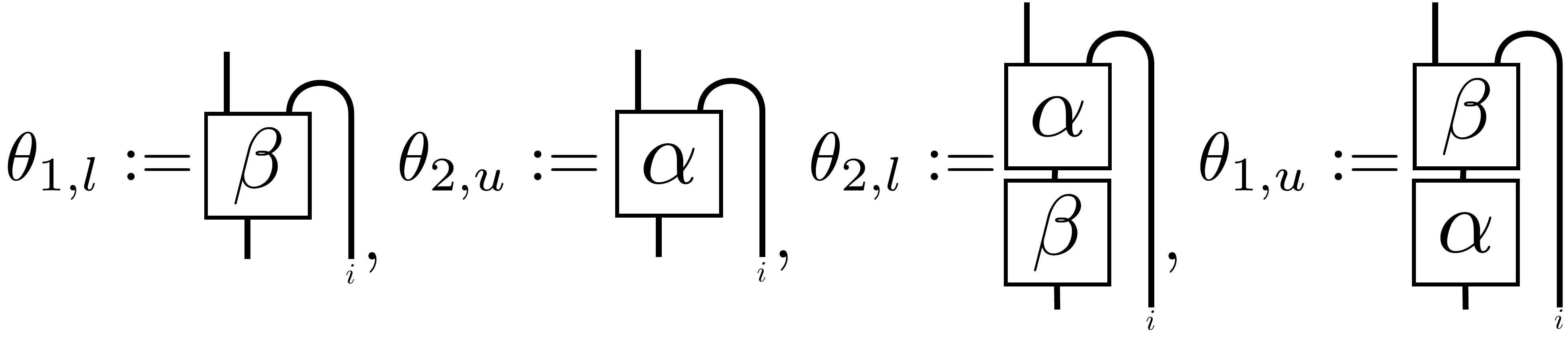}
\end{figure}
\begin{figure}[H]
	\centering
	\includegraphics[width=1\linewidth]{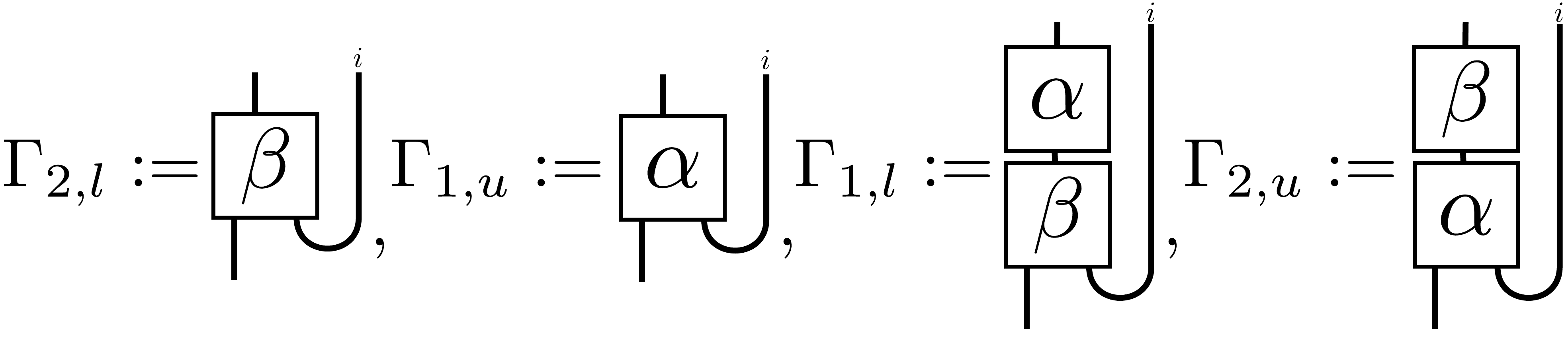}
\end{figure}

\end{prop}

\begin{proof}
We give a proof for $\theta$, with the proof for $\Gamma$ following similarly.\\
From how these maps act on the weight spaces, we can conclude that they must contain $\alpha$ and $\beta$. As the diagrams will have $2p-i-1$ points at the top and $2p+i-1$ points at the bottom, then considering the properties of $\alpha$ and $\beta$, we see that, up to some constant, the diagrams for $\theta_{1,l}$ and $\theta_{2,u}$ must be as given.\\
Determining the diagrams for $\theta_{2,l}$ and $\theta_{1,u}$ is more complicated. From the expected weight space action and number of points along the top and bottom, we find that $\theta_{2,l}$ can be a sum over diagrams of the form $(1^{\otimes j}\otimes\beta\otimes 1^{\otimes(i-j)})(\alpha\otimes 1^{\otimes i})(1^{\otimes (2p-i-1)}\otimes\cap^{\otimes i})$ for $0\leq j\leq i$, plus Temperley-Lieb elements.

To show which one of these diagrams is correct, we need the following. Denote the sum in relation \ref{eq:15} by $S_{\alpha}$, and consider the following diagram:
\begin{figure}[H]
	\centering
	\includegraphics[width=0.15\linewidth]{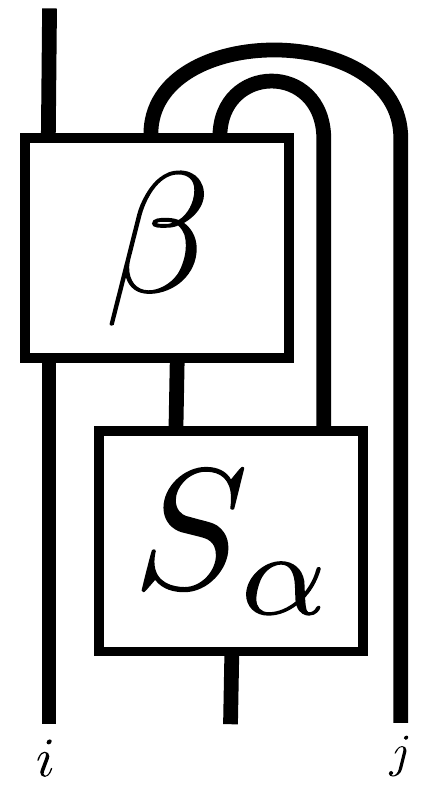}
\end{figure}
with $k_{1}=1,k_{2}=0$. It simplifies to give the following relation:
\begin{figure}[H]
		\centering
		\includegraphics[width=0.7\linewidth]{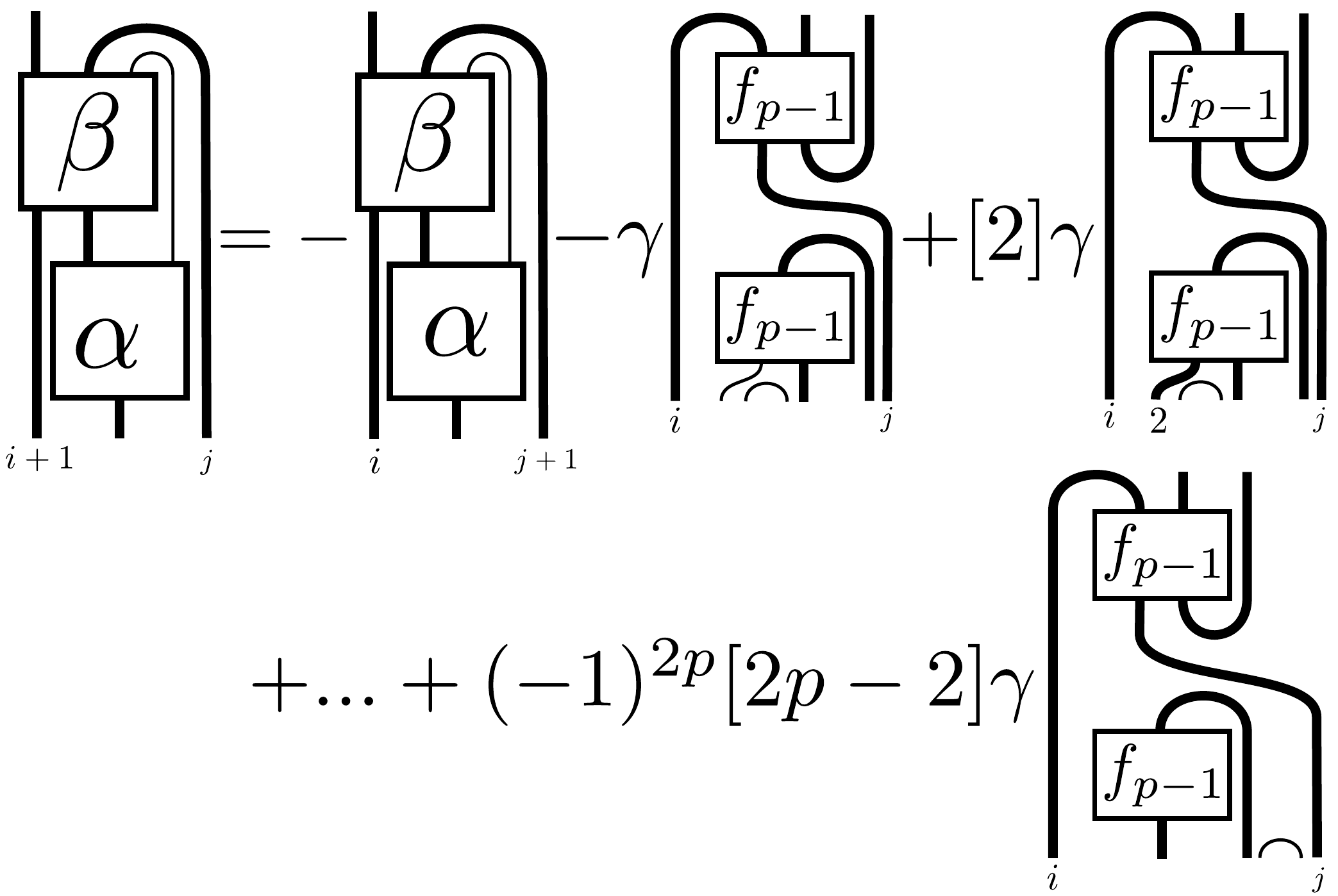}
\end{figure}
with $0\leq i+j\leq 2p-2$. Hence shifting a box sideways is equivalent to multiplying by $-1$ and adding some Temperley-Lieb elements. It follows immediately from this that we can write:
\begin{figure}[H]
		\centering
		\includegraphics[width=1\linewidth]{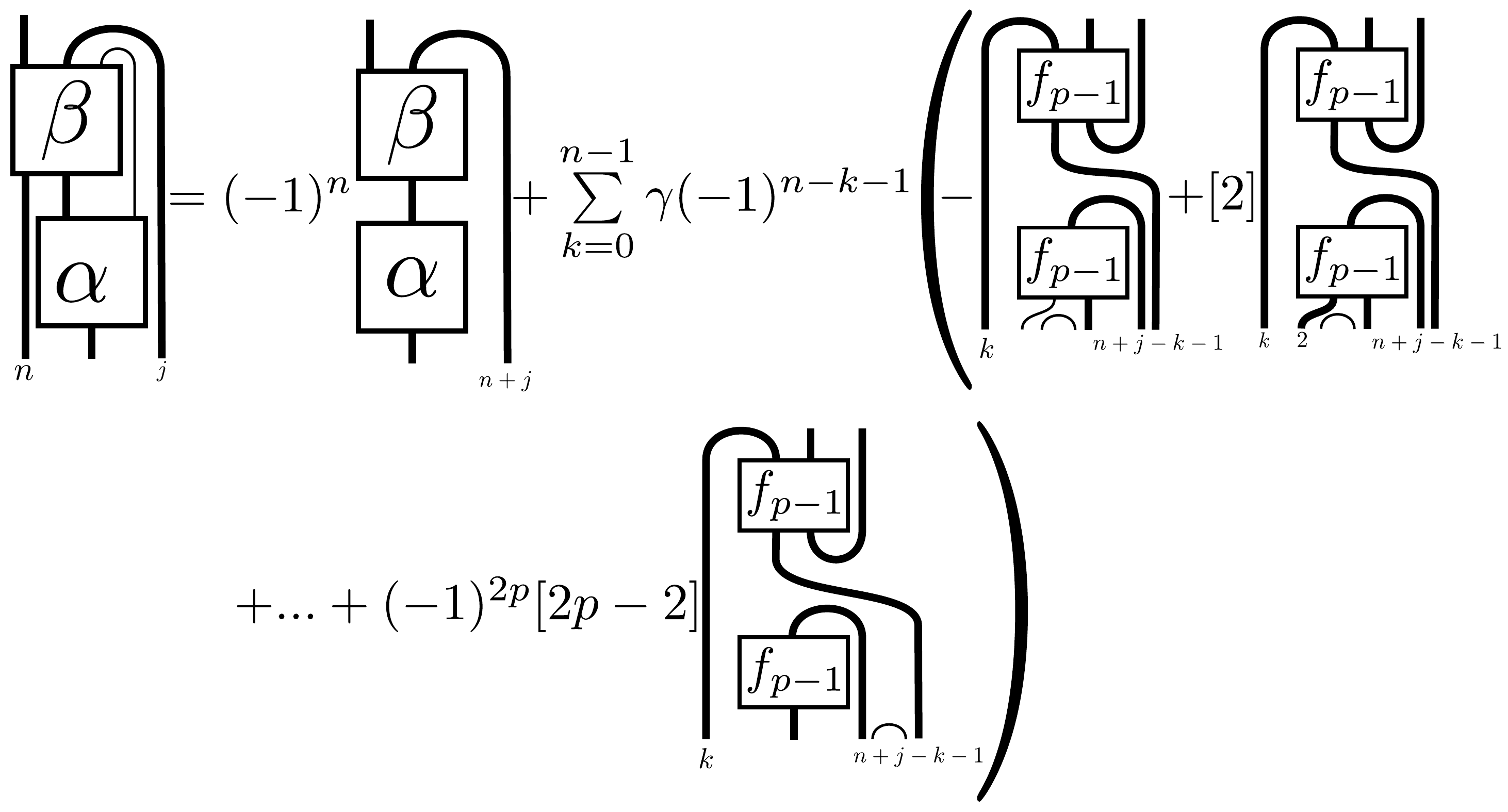}
\end{figure}
Hence we just need to show that any Temperley-Lieb elements in the map must be equal to zero. There are two cases to consider, TL elements appearing in the above diagram, and other TL elements.\\
For the first case, consider the formula for the projections onto $\mathcal{P}^{-}_{i}$ given in Section \ref{ind proj}. From how they act on the module bases, taking the composition of the homomorphism $\theta$ and the projection should be equal to the homomorphism. Noting this, and that each term in the projection contains $\alpha$ or $\beta$ at the left-hand side, then any of the Temperley-Lieb elements in the above sum, acting on the projection, will give zero. Hence it follows that for the composition of the homomorphism and the projection to equal the homomorphism, then the only term in the homomorphism containing $\alpha$ and $\beta$, must be $(\beta\alpha\otimes 1^{\otimes i})(1^{\otimes (2p-i-1)}\otimes\cap^{\otimes i})$.\\
To rule out other TL elements, we note that the TL element acting on the projection must cause the partial trace of $\alpha\beta$ to give a TL element, otherwise the composition of the homomorphism and projection would not equal the homomorphism. However if it does cause the partial trace, then the resulting $(2p-i-1\rightarrow 2p+i-1)$ TL element would have at most $i-1$ through strings. For a TL element acting on the projection to be non-zero, it must have at least $2p-i-1$ through strings. Hence for the TL element to appear in the homomorphism, we need $2p-i-1\leq i-1$, which occurs when $i\geq p$. As the homomorphism is only for $i\leq p-1$, this shows that there are no TL elements in the homomorphism, and so the diagrams for $\theta_{2,l}$ and $\theta_{1,u}$ are as given.
\end{proof}

By considering appropriate compositions of the previous homomorphisms, we get:
\begin{corr}
Up to a constant, the second endomorphism on the two copies of $\mathcal{P}^{-}_{i}$ in $X^{\otimes 3p-i-1}$ is given by:
\begin{figure}[H]
	\centering
	\includegraphics[width=0.3\linewidth]{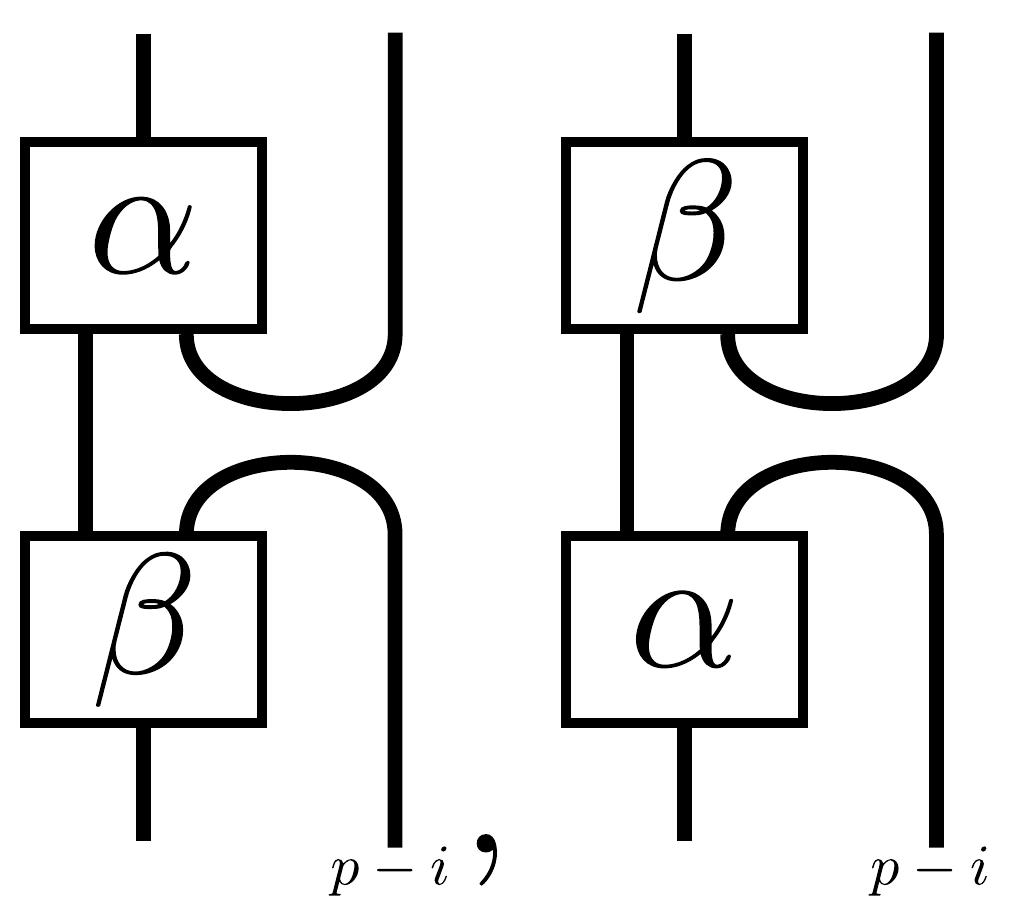}
\end{figure}
\end{corr}
Considering other compositions of the diagrammatic morphisms, we should expect the following relations to hold:
\begin{figure}[H]
	\centering
	\includegraphics[width=0.7\linewidth]{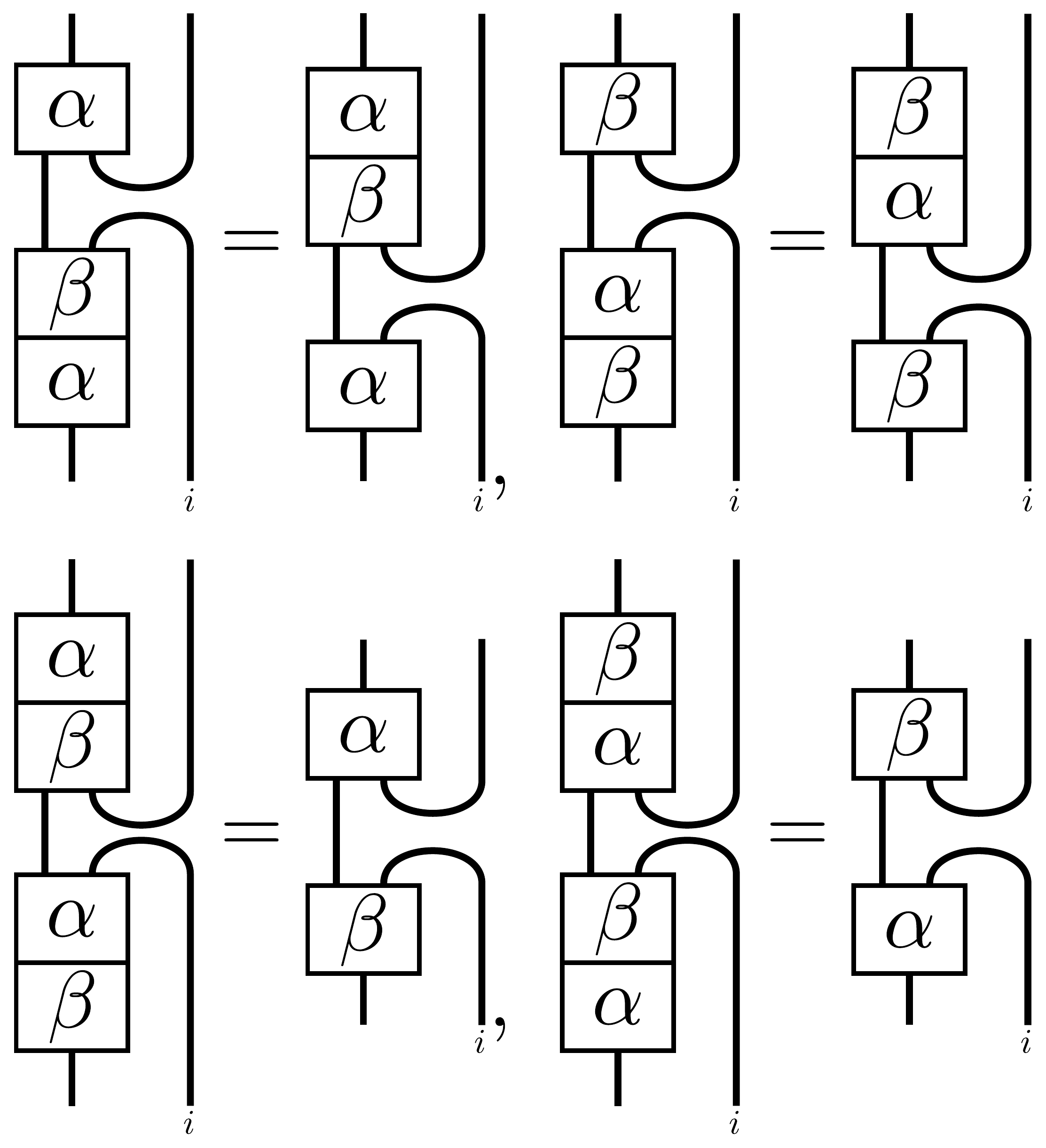}
\end{figure}
It is straightforward to see that, for fixed $i$, these are all equivalent. These relations can be proven by inserting relation \ref{eq:15}, denoted $S_{\alpha}$, into the following diagram and simplifying any other diagrams that appear:
\begin{figure}[H]
	\centering
	\includegraphics[width=0.3\linewidth]{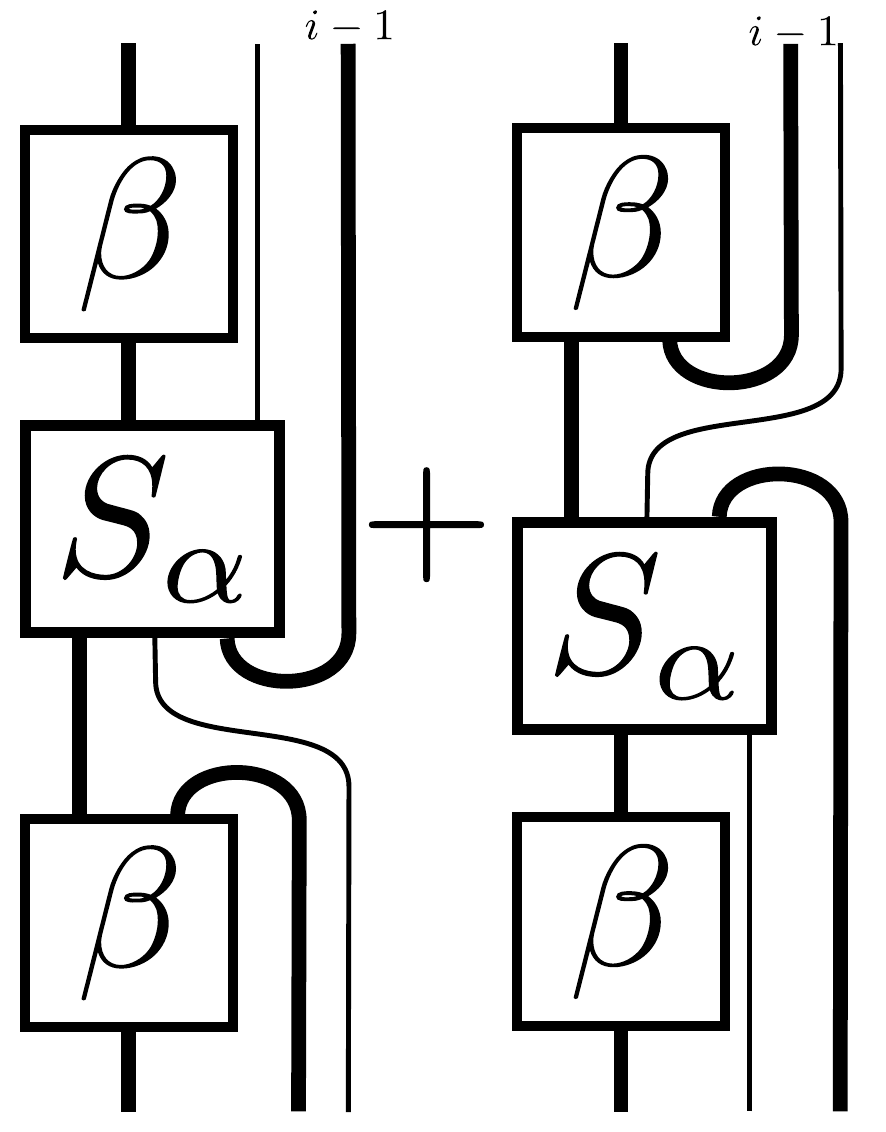}
\end{figure}

\titleformat{\section}{\normalfont\Large\bfseries}{\appendixname~\thesection.}{1em}{}
\begin{appendices}
	\numberwithin{equation}{section}
	\section{Combinatorial Relations.}
	There are a number of generalized relations for $\bar{U}_{q}(\mathfrak{sl}_{2})$ and its action on $X^{\otimes n}$, which we record here.\\
	The quantum group $\bar{U}_{q}(\mathfrak{sl}_{2})$ and its relations can be used to give the following generalized conditions:
	\begin{align}
	\Delta^{k}(K)&=K^{\otimes k+1} \label{eq:A1}\\
	\Delta^{k}(E)&=\sum\limits_{i=0}^{k}(1^{\otimes i})\otimes E\otimes (K^{\otimes (k-i)}) \label{eq:A2}\\
	\Delta^{k}(F)&=\sum\limits_{i=0}^{k}\big((K^{-1})^{\otimes i}\big)\otimes F\otimes(1^{\otimes (k-i)}) \label{eq:A3}\\
	EF^{k}&=F^{k}E+(\frac{[k]}{q-q^{-1}})\big(q^{1-k}F^{k-1}K-q^{k-1}F^{k-1}K^{-1}\big) \label{eq:A4}\\
	FE^{k}&=E^{k}F+(\frac{[k]}{q-q^{-1}})\big(q^{1-k}E^{k-1}K^{-1}-q^{k-1}E^{k-1}K\big) \label{eq:A5}\\
	\Delta E^{k}&=\sum\limits_{i=0}^{k}\lambda_{i,k}E^{i}\otimes K^{i}E^{k-i} \label{eq:A6}\\
	\Delta F^{k}&=\sum\limits_{i=0}^{k}\lambda_{i,k}K^{-i}F^{k-i}\otimes F^{i} \label{eq:A7}\\
	\lambda_{i,k}&=q^{(i^{2}-ik)}\frac{([k]!)}{([i]!)([k-i]!)} \label{eq:A8}
	\end{align}
	The $\bar{U}_{q}(\mathfrak{sl}_{2})$ action on the basis elements satisfies the following:
	\begin{align}
	E^{n}\rho_{i_{1},...,i_{n},z}&=q^{\big(nz-\frac{1}{2}(n^{2}-n)-(\sum\limits_{j=1}^{n}i_{j})\big)}([n]!)x_{0,z} \label{eq:A9}\\
	F^{z-n}\rho_{i_{1},...,i_{n},z}&=q^{\big(nz-\frac{1}{2}(n^{2}-n)-(\sum\limits_{j=1}^{n}i_{j})\big)}([z-n]!)x_{z,z} \label{eq:A10}\\
	F^{k}x_{0,z}&=\sum\limits_{1\leq i_{j}\leq z}q^{\big(\frac{1}{2}(k^{2}+k)-(\sum\limits_{j=1}^{k}i_{j})\big)}([k]!)\rho_{i_{1},...,i_{k},z} \label{eq:A11}\\
	E^{k}x_{z,z}&=\sum\limits_{1\leq i_{j}\leq z}q^{\big(\frac{1}{2}(z-k)(z-k+1)-(\sum\limits_{j=1}^{z-k}i_{j})\big)}([k]!)\rho_{i_{1},...,i_{z-k},z} \label{eq:A12}\\
	E^{k}x_{z+1,z+1}&=[k](E^{k-1}x_{z,z})\otimes\nu_{0}+q^{-k}(E^{k}x_{z,z})\otimes\nu_{1} \label{eq:A13}\\
	&=q^{k-z-1}[k]\nu_{0}\otimes(E^{k-1}x_{z,z})+\nu_{1}\otimes(E^{k}x_{z,z}) \label{eq:A14}\\
	F^{k}x_{0,z+1}&=(F^{k}x_{0,z})\otimes\nu_{0}+q^{k-z-1}[k](F^{k-1}x_{0,z})\otimes\nu_{1} \label{eq:A15}\\
	&=q^{-k}\nu_{0}\otimes(F^{k}x_{0,z})+[k]\nu_{1}\otimes(F^{k-1}x_{0,z}) \label{eq:A16}
	\end{align}
	These come from considering all contributions to the coefficient as different orderings of the integers $i_{1},...,i_{n}$, where each ordering describes the order in which the zero's appeared. For the standard ordering with $i_{1}<i_{2}<...<i_{n}$, its contribution to the coefficient is just $q^{-z+i_{1}}q^{-z+i_{2}}...q^{-z+i_{n}}=q^{(-nz+\sum\limits_{j=1}^{n}i_{j})}$. Interchanging two integers in the ordering multiplies this by $q^{\pm 2}$, and the coefficient comes from considering all possible permutations.\\
	\\
	For  integers $1\leq i_{1}<i_{2}<...<i_{n}\leq z$, we have:
	\begin{align}
	\xi_{n,z}:=&\sum\limits_{1\leq i_{j}\leq z}q^{-2(\sum\limits_{j=1}^{n}i_{j})}=q^{-n-nz}\frac{([z]!)}{([n]!)([z-n]!)} \label{eq:A17}\\
	\xi_{n,z}=& q^{-2z}\xi_{n-1,z-1}+\xi_{n,z-1}
	\end{align}
	where the recurrence relation comes from considering the two cases in $\xi_{n,z}$, when $i_{n}=z$ and when $i_{n}\neq z$.  
\end{appendices}
\subsection*{Acknowledgements}
The author's research was supported by EPSRC DTP grant EP/K502819/1. The author would like to thank the Isaac Newton Institute for Mathematical Sciences, Cambridge, for support and hospitality during the programme Operator Algebras: Subfactors and their applications, where work on this paper was undertaken. This work was supported by EPSRC grant no EP/K032208/1. This work was partially supported by ISF grant 2095/15 and the Center for Advanced Studies in Mathematics in Ben Gurion University.
\bibliography{references}
\bibliographystyle{abbrv}

\end{document}